\documentclass[11pt,a4paper]{article}
\usepackage{amsmath,amsfonts,amssymb,latexsym,graphics,epsfig}
\usepackage{color}
\usepackage{amsthm}
\usepackage{graphicx,url}
\usepackage{hyperref}
\usepackage[english]{babel}
\usepackage{epsfig}
\usepackage{enumerate}
\usepackage{graphicx}
\usepackage{floatrow}
\usepackage{float}
\usepackage{pifont}
\usepackage{caption}
\usepackage{subfig}
\usepackage{fullpage}
\usepackage{fancyhdr}
\usepackage[arrow,frame,matrix]{xy}

\newtheorem{thm}{Theorem}[section]
\newtheorem{lem}[thm]{Lemma}
\newtheorem{pro}[thm]{Proposition}
\newtheorem{de}[thm]{Definition}
\newtheorem{re}[thm]{Remark}

\newtheorem{cor}[thm]{Corollary}
\newtheorem{claim}[thm]{Claim}

\newtheorem{problem}{Problem}

\def\be{\begin{equation}}
\def\ee{\end{equation}}
\def\bea{\begin{eqnarray}}
\def\eea{\end{eqnarray}}
\numberwithin{equation}{section}

\begin{document}
\title{The integrally representable trees of norm $3$}
\author{Jack H. Koolen\thanks{J.H. Koolen is partially supported by the National Natural Science Foundation of China (Grant No. 11471009 and Grant No. 11671376).}~, Masood Ur Rehman\thanks{M.U. Rehman is supported by the Chinese Scholarship Council at USTC, China.}~, Qianqian Yang}
\maketitle
\date{}
\begin{abstract}
In this paper, we determine the integrally representable trees of norm $3$.
\end{abstract}



\section{Introduction}
The connected graphs with spectral radius two were classified by Smith (see \cite[Theorem $3.2.5$]{drg} and \cite{smith}) and they are as follows (where the index equals $|V(G)|-1$):
\begin{equation*}
  \begin{array}{ll}
   \widetilde{A}_m~(m\geq2):\quad& \raisebox{-0.7ex}{\includegraphics[scale=0.8]{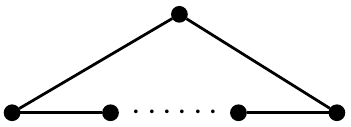}};\\
   \widetilde{D}_m~(m\geq4):\quad& \raisebox{-0.7ex}{\includegraphics[scale=0.8]{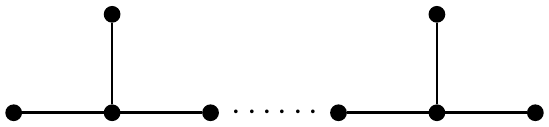}};\\
   \widetilde{E}_6:\quad &\raisebox{-0.7ex}{\includegraphics[scale=0.8]{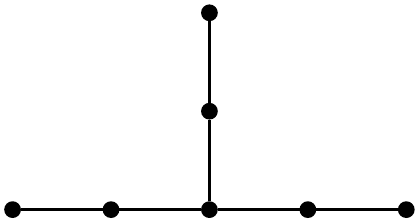}}; \\
   \widetilde{E}_7:\quad&\raisebox{-0.7ex}{\includegraphics[scale=0.8]{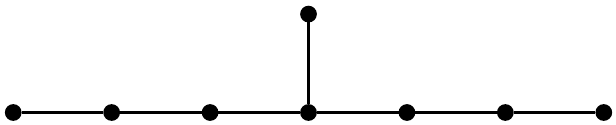}};\\
   \widetilde{E}_8:\quad&\raisebox{-0.7ex}{\includegraphics[scale=0.8]{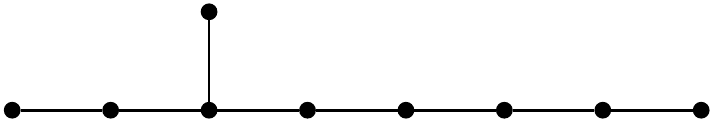}}.
\end{array}
\end{equation*}

Moreover, each connected graph with spectral radius at most $2$ is a subgraph of one of the graphs above, and each connected graph with radius at least $2$ contains one of the above graphs as an induced subgraph. Note that except for $\widetilde{A}_m$, they are all trees. And this classification is closely related to the classification of the irreducible root lattices. In this paper, we will look at trees with spectral radius at most $3$ and hence with smallest eigenvalue at least $-3$. Note that there are infinite trees with spectral radius $3$. For example, the family of trees in Figure \ref{tree}:
\begin{figure}[H]
\centering
\includegraphics[scale=1]{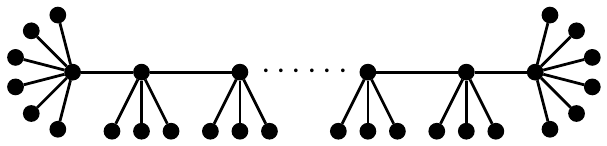}
\caption{}
\label{tree}
\end{figure}

Now let us look at graphs with smallest eigenvalue at least $-3$. Let $G$ be a graph with $m$ vertices, adjacency matrix $A$ and smallest eigenvalue at least $-3$. Then $A + 3I$ is positive semidefinite and hence the Gram matrix of a set of vectors $S =\{{\bf x}_1, \ldots, {\bf x}_m\}$. Then the lattice generated by $S$ is an integral lattice generated by vectors with norm $3$ ($=$ length $\sqrt{3}$). So understanding graphs with smallest eigenvalue at least $-3$ helps us to understand the integral lattices generated by norm $3$ vectors. And the trees with smallest eigenvalue at least $-3$ are the simplest among them.

In this paper, we will look at a restricted class of trees with spectral radius at most $3$, namely, those that are integrally representable of norm $3$. We say that the graph $G$ with adjacency matrix $A$ is integrally representable of norm $t$ if there exists a matrix $N$ with integral entries such that $A+tI = N^TN$. Note that an integrally representable graph $G$ with norm $t$ has smallest eigenvalue at least $-t$. In this paper, we will classify the integrally representable trees of norm $3$. For some other research on integral lattices, we refer to \cite{CS}, \cite{PP}, \cite{PP2} and \cite{SH}. For some earlier work on integral lattices generated by vectors of norm $3$, we refer to \cite{Neumaier} and \cite[p.111]{drg}.

The paper is organized as follows. In Section \ref{construction}, we will give a construction to construct integrally representable trees of norm $3$. Moreover, in Section \ref{sectree}, we will show that with the construction we can construct all integrally representable trees of norm $3$. We will use Hoffman graphs as our main tool. Properties of Hoffman graphs and other definitions and preliminaries will be discussed in Section \ref{secprelim}. We will conclude the paper with a discussion about seedlings.

\section{Preliminaries}\label{secprelim}
\subsection{Graphs}
Throughout the paper we will consider only undirected graphs without loops or multiple edges. Let $G=(V(G),E(G))$ be a graph on $m$ vertices. Recall that the adjacency matrix $A$ of $G$ is an $m\times m$ matrix indexed by the vertices of $G$ such that $A_{xy}=1$ if the vertex $x$ is adjacent to the vertex $y$ and $A_{xy}=0$ otherwise. We write $x\sim y$ if the vertices $x$ and $y$ are adjacent and $x\not\sim y$ if they are not adjacent. By the eigenvalues of $G$ we mean the eigenvalues of $A$ and denote $\lambda_{\min}(G)$ the smallest eigenvalue of $G$.

Let $t$ be a positive integer. The graph $G$ is called \emph{integrally representable of norm $t$}, if there exists a map $\varphi : V(G) \rightarrow \mathbb{Z}^n$ for some positive integer $n$ such that

\[
(\varphi(x),\varphi(y) )=
\left\{
  \begin{array}{ll}
    $t$ &\text{ if $x=y$}, \\
    1 &\text{ if $x \sim y$}, \\
    0 &\text{ otherwise}. \\
  \end{array}
\right.
\]

We also call the map $\varphi$ an integral representation of $G$ of norm $t$. When $t=3$, we have $\varphi(x)_i\in\{0,\pm1\}$ for all $x\in V(G)$ and $i\in\{1,\ldots,m\}$.

Let $\varphi$ be an integral representation of $G$ of norm $t$. Assume that $G$ has vertices $x_1,...,x_m$ and let $N$ be the matrix whose ${x_i}^{th}$ column is equal to $\varphi(x_i)$ for $i=1,...,m$. Then $N^TN=A+tI$, where $A$ is the adjacency matrix of $G$. This means $\lambda_{\min}(G)\geq -t$.

A tree $T$ is a connected undirected graph without cycles. If the tree $T$ has an integral representation of norm $3$, we call it an integrally representable tree of norm $3$.

\subsection{Matrices}
 In this subsection, we will introduce two main theorems about matrices and its eigenvalues: eigenvalue interlacing and the Perron-Frobenius Theorem.
\begin{thm}(\cite[Interlacing]{H1995})\label{interlacing}
Let $A$ be a real symmetric $n\times n$  matrix with
eigenvalues $\theta_1\ge\cdots\ge \theta_n$. For some $m<n$,
let $S$ be a real $n\times m$ matrix with orthonormal columns,
$S^{T}S=I$, and consider the matrix $B=S^{T}AS$,
with eigenvalues $\mu_1\ge\cdots\ge \mu_m$. Then,
\begin{enumerate}
\item  the eigenvalues of $B$ interlace those of
    $A$, that is,
$\theta_i\ge \mu_i\ge \theta_{n-m+i},~i=1,\ldots, m;$

\item if there exists an integer $j\in \{1,2,\ldots,m\}$ such that $\theta_{i}=\mu_{i}$ for $1\leq i\leq j$ and $\theta_{n-m+i}=\mu_{i}$ for $j+1\leq i \leq m$, then the interlacing is tight and $SB=AS$.
\end{enumerate}
\end{thm}
For a matrix $A$, we write $A>0$ (resp. $A\geq0$) when all its entries are positive (nonnegative). Let $A$ be a real $n\times n$ matrix with nonnegative entries. The matrix $A$ is called \emph{irreducible} if there exists an integer $p$ such that $A^p>0$. Note that the matrix $A$ is irreducible if and only if the directed graph $G_A$ with vertices $\{1,\ldots,n\}$ and edges $ij$ whenever $A_{ij}>0$ is strongly connected. Now we recall the Perron-Frobenius Theorem as follows:

 \begin{thm}(\cite[Theorem $3.1.1$]{drg})\label{perronfrobenius}
 Let $A$ be a nonnegative irreducible matrix, then the following holds:
 \begin{enumerate}
 \item There is a real number $\theta_0$ and a real vector ${\bf x}_0$ with $A{\bf x}_0=\theta_0{\bf x}_0$, $\theta_0>0,{\bf x}_0>0$. If a vector ${\bf x}\geq0,{\bf x}\neq0$ and $A{\bf x}\geq\theta {\bf x}$, then $\theta\leq\theta_0$.
 \item The eigenvalue $\theta_0$ of $A$ has geometric and algebraic multiplicity one.
 \item For each eigenvalue $\theta$ of $A$ we have $|\theta|\leq\theta_0$.
 \item If a vector ${\bf x}\geq0,{\bf x}\neq0$ and $A{\bf x}\leq\theta {\bf x}$, then ${\bf x}>0$ and $\theta\geq\theta_0$; moreover, $\theta=\theta_0$ if and only if $A{\bf x}=\theta {\bf x}$.
 \item If $0\leq B\leq A$ and $B\neq A$, then every eigenvalue $\mu$ of $B$ satisfies $|\mu|<\theta_0$.
 \end{enumerate}
 \end{thm}

\subsection{Hoffman graphs}
Now we are in the position to give definitions and preliminaries of Hoffman graphs. Hoffman graphs were introduced by Woo and Neumaier \cite{woo}. Most of the material of this section comes from \cite{Koolen} and \cite{woo}.

\begin{de} A Hoffman graph $\mathfrak{h}$ is a pair $(H,l)$ where $H=(V,E)$ is a graph and $l:V \rightarrow \{f,s\}$ is a labeling map satisfying the following conditions:
\begin{enumerate}
\item every vertex with label $f$ is adjacent to at least one vertex with label $s$;
\item  vertices with label $f$ are pairwise non-adjacent.
\end{enumerate}
\end{de}
We call a vertex with label $s$ a \emph{slim vertex}, and a vertex with label $f$ a \emph{fat vertex}. We denote
by $V_s(\mathfrak{h})$ (resp. $V_f(\mathfrak{h})$) the set of slim (resp. fat) vertices of $\mathfrak{h}$.

\vspace{0.1cm}
For a vertex $x$ of $\mathfrak{h}$, we define $N_{\mathfrak{h}}^{s}(x)$ (resp. $N_{\mathfrak{h}}^{f}(x)$) the set of slim (resp. fat) neighbors of $x$ in $\mathfrak{h}$.
If every slim vertex of the Hoffman graph $\mathfrak{h}$ has a fat neighbor, then we call $\mathfrak{h}$ \emph{fat}. In a similar fashion, we define $N^{f}_\mathfrak{h}(x_1,x_2)$ to be the set of common fat neighbors of two slim vertices $x_1$ and $x_2$ in $\mathfrak{h}$ and $N^{s}_\mathfrak{h}(f_1,f_2)$ to be the set of common slim neighbors of two fat vertices $f_1$ and $f_2$ in $\mathfrak{h}$.

\vspace{0.1cm}
The \emph{slim graph} of the Hoffman graph $\mathfrak{h}$ is the subgraph of $H$ induced on $V_s(\mathfrak{h})$.

\vspace{0.1cm}
A Hoffman graph $\mathfrak{h_1 }= (H_1, l_1)$ is called an \emph{induced Hoffman subgraph} of $\mathfrak{h}=(H, l)$, if $H_1$ is an induced subgraph of $H$ and $l_1(x) = l(x)$ holds for all vertices $x$ of $H_1$.

\vspace{0.1cm}
Let $W$ be a subset of $V_s(\mathfrak{h})$. An \emph{induced Hoffman subgraph of $\mathfrak{h}$ generated by $W$}, denoted by $\langle W\rangle_{\mathfrak{h}}$, is the Hoffman subgraph of $\mathfrak{h}$ induced by $W \cup\{f\in V_f(\mathfrak{h})~|f \sim w \text{ for some }w\in W \}$.

\vspace{0.1cm}
Note that any graph can be considered as a Hoffman graph with only slim vertices, and vice versa. We will not distinguish between Hoffman graphs with only slim vertices and graphs.
\begin{de}
A Hoffman graph $\mathfrak{t}= (T, l)$ is called \emph{tree-like} if the graph $T$ is a tree.
\end{de}

\begin{de} Two Hoffman graphs $\mathfrak{h}=(H, l)$ and $\mathfrak{h}'=(H',l')$ are \emph{isomorphic} if there exists an isomorphism $\sigma:H\rightarrow H'$ such that $\sigma$ preserves the labeling.
\end{de}

\begin{de} For a Hoffman graph $\mathfrak{h}=(H,l)$, let $A$ be the adjacency matrix of $H$
\begin{eqnarray*}
A=\left(
\begin{array}{cc}
A_s  & C\\
C^{T}  & O
\end{array}
\right)
\end{eqnarray*}
in a labeling in which the fat vertices come last. The \emph{special matrix} $Sp(\mathfrak{h})$ of $\mathfrak{h}$ is the real symmetric matrix $Sp(\mathfrak{h}):=A_s-CC^{T}.$ The eigenvalues of $\mathfrak{h}$ are the eigenvalues of $Sp(\mathfrak{h})$. We denote $\lambda_{\min}(\mathfrak{h})$ the smallest eigenvalue of $\mathfrak{h}$.
\end{de}

\begin{de} Let $\mathfrak{h}$ be a Hoffman graph and let $n$ be a positive integer. A mapping $\phi : V(\mathfrak{h}) \rightarrow \mathbb{R}^{n}$ satisfying
\[(\phi(x),\phi(y))=
\begin{cases}
t &\text{if $x=y\in V_{s}(\mathfrak{h})$,}\\
1 &\text{if $x=y\in V_{f}(\mathfrak{h})$,}\\
1 &\text{if $x$ and $y$ are adjacent,}\\
0 &\text{otherwise,}
\end{cases}
\]
is called a \emph{representation of norm $t$}, where $(,)$ is the standard inner product on $\mathbb{R}^{n}.$ We denote by $\Lambda(\frak{h},t)$ the lattice generated by $\{\phi(x)\mid x \in V(\frak{h})\}$. Note that the isomorphism class of $\Lambda(\frak{h},t)$ depends only on $t$, and is independent of $\phi$, justifying the notation.
\end{de}

\begin{de}\label{reducedrep}
Let $\mathfrak{h}$ be a Hoffman graph and let $n$ be a positive integers. A mapping $\psi : V_s(\mathfrak{h}) \rightarrow \mathbb{R}^{n}$ satisfying
\[(\psi(x),\psi(y))=
\begin{cases}
t-|N^{f}_{\mathfrak{h}}(x)| &\text{if $x=y$,}\\
1-|N^{f}_{\mathfrak{h}}(x,y)| &\text{if $x \sim y$,}\\
-|N^{f}_{\mathfrak{h}}(x,y)| &\text{otherwise,}
\end{cases}
\]
is called a \emph{reduced representation of norm $t$}, where $(,)$ is the standard inner product on $\mathbb{R}^{n}$. We denote by $\Lambda^{red}(\frak{h},t)$ the lattice generated by $\{\psi(x)\mid x \in V_s(\frak{h})\}$. Note that the isomorphism class of $\Lambda^{red}(\frak{h},t)$ depends only on $t$, and is independent of $\psi$, justifying the notation.
\end{de}
\begin{lem}(\cite[Theorem $2.8$]{Koolen})\label{relation}
For a Hoffman graph $\mathfrak{h}$, the following conditions are equivalent:
\begin{enumerate}
\item $\mathfrak{h}$ has a representation of norm $t$;
\item $\mathfrak{h}$ has a reduced representation of norm $t$;
\item $\lambda_{\min}(\mathfrak{h})\geq -t$.
\end{enumerate}
\end{lem}

The following lemma shows how to construct a reduced representation of norm $t$ from a given representation of norm $t$.

\begin{lem}(\cite[Lemma $2.7$]{Koolen})\label{reprerelation}
Let $\mathfrak{h}$ be a Hoffman graph having a representation of norm $t$. Then $\mathfrak{h}$ has a reduced representation of norm $t$, and $\Lambda(\frak{h},t)$ is isomorphic to $\Lambda^{red}(\frak{h},t)\bigoplus\mathbb{Z}^{|V_f(\mathfrak{h})|}$ as a lattice.
\end{lem}

A Hoffman graph $\mathfrak{h}$ is called integrally representable of norm $t$, if $\mathfrak{h}$ has an integral represention $\phi: V(\mathfrak{h}) \rightarrow \mathbb{Z}^{n}$ of norm $t$ for some $n$. From Lemma \ref{reprerelation}, we find that a Hoffman graph $\mathfrak{h}$ is integrally representable of norm $t$ is also equivalent to that $\mathfrak{h}$ has an integral reduced representation $\psi : V_s(\mathfrak{h}) \rightarrow \mathbb{Z}^{n}$ of norm $t$ for some $n$.

An \emph{edge-signed graph} $\mathcal{S}$ is a pair $(S,\text{sgn})$ of a graph $S$ and a map $\text{sgn}: E(S)\rightarrow \{+,-\}$. Let $V(\mathcal{S})=V(S)$, $E^+(\mathcal{S}):=\text{sgn}^{-1}(+)$ and $E^-(\mathcal{S}):=\text{sgn}^{-1}(-)$. Each element in $E^+(\mathcal{S})$ (resp. $E^-(\mathcal{S})$) is called a $(+)$-edge (resp. $(-)$-edge) of $\mathcal{S}$. We represent an edge-signed graph $\mathcal{S}$ also by the triple $(V(\mathcal{S}), E^+(\mathcal{S}), E^-(\mathcal{S}))$.

For an edge-signed graph $\mathcal{S}$, we define its \emph{signed adjacency matrix} $B(\mathcal{S})$ by

\[B(\mathcal{S})_{xy}=
\begin{cases}
1 &\text{if $xy\in E^{+}(\mathcal{S})$,}\\
-1 &\text{if $xy\in E^{-}(\mathcal{S})$,}\\
0 &\text{otherwise.}
\end{cases}
\]

A \emph{switching} at vertex $x$ is swapping the sign of each edge incident to $x$. Two edge-signed graphs $\mathcal{S}$ and $\mathcal{S}'$ are called \emph{switching  equivalent}, if there exists a subset $U\subset V(\mathcal{S})$ such that $\mathcal{S}'$ is isomorphic to the edge-signed graph obtained from $\mathcal{S}$ by switching at each vertex in $U$. Note that switching equivalence is an equivalence relation that preserves eigenvalues.

Let $\psi$ be a reduced representation of norm $t$ of $\mathfrak{h}$. The \emph{special graph} of $\mathfrak{h}$ is the edge-signed graph $\mathcal{S}(\mathfrak{h}):=(V(\mathcal{S}(\mathfrak{h})), E^+(\mathcal{S}(\mathfrak{h})), E^-(\mathcal{S}(\mathfrak{h}))),$ where $V(\mathcal{S}(\mathfrak{h}))=V_s(\mathfrak{h})$ and
 $$E^+(\mathcal{S}(\mathfrak{h}))=\big\{\{x,y\}\mid x,y\in V_s(\mathfrak{h}),x\neq y,\text{sgn}(\psi(x),\psi(y))=+\big\},$$
 $$E^-(\mathcal{S}(\mathfrak{h}))=\big\{\{x,y\}\mid x,y\in V_s(\mathfrak{h}),x\neq y,\text{sgn}(\psi(x),\psi(y))=-\big\}.$$

The \emph{special $\varepsilon$-graph} of $\mathfrak{h}$ is the graph $S^\epsilon(\mathfrak{h})=(V_s(\mathfrak{h}),E^\epsilon(\mathcal{S}(\mathfrak{h})))$ for $\epsilon\in\{+,-\}$.

Now, we introduce \emph{direct sums} of Hoffman graphs.

\begin{de}\label{directsummatrix}
Let $\mathfrak{h}^1$ and $\mathfrak{h}^2$ be two Hoffman graphs. We call a Hoffman graph $\mathfrak{h}$ the direct sum of $\mathfrak{h}^1$ and $\mathfrak{h}^2$, denoted by $\mathfrak{h} = \mathfrak{h}^1 \oplus \mathfrak{h}^2$, if $\mathfrak{h}$ satisfies the following condition:

There exists a partition $\big\{V_s^1(\mathfrak{h}),V_s^2(\mathfrak{h})\big\}$ of $V_s(\mathfrak{h})$ such that induced Hoffman subgraphs generated by $V_s^i(\mathfrak{h})$ are $\mathfrak{h}^i$ for $i=1,2$ and
$Sp(\mathfrak{h})=
\begin{pmatrix}
Sp(\mathfrak{h}^1) & 0 \\
0 & Sp(\mathfrak{h}^2)
\end{pmatrix}
$ with respect to the partition $\big\{V_s^1(\mathfrak{h}),V_s^2(\mathfrak{h})\big\}$ of $V_s(\mathfrak{h})$.
\end{de}

Clearly, by definition, the direct sum is associative, so that the sum $\oplus_{i=1}^{r}\mathfrak{h}^i$ is well-defined. We can check that $\mathfrak{h}$ is a direct sum of two non-empty Hoffman graphs if and only if $Sp(\mathfrak{h})$ is a block matrix with at least $2$ blocks. If $\mathfrak{h} = \mathfrak{h}^1 \oplus \mathfrak{h}^2$ for some non-empty Hoffman subgraphs $\mathfrak{h}^1$ and $\mathfrak{h}^2$, then we call $\mathfrak{h}$ \emph{decomposable}. Otherwise, $\mathfrak{h}$ is called \emph{indecomposable}.

The following lemma gives a combinatorial interpretation of a direct sum of two Hoffman graphs.

\begin{lem}\label{directsumcombi}
Let $\mathfrak{h}$ be a Hoffman graph and $\mathfrak{h}^1$ and $\mathfrak{h}^2$ be two induced Hoffman subgraphs of $\mathfrak{h}$. The Hoffman graph $\mathfrak{h}$ is the direct sum of $\mathfrak{h}^1$ and $\mathfrak{h}^2$ if and only if $\mathfrak{h}^1$, $\mathfrak{h}^2$, and $\mathfrak{h}$ satisfy the following conditions:
\begin{enumerate}
\item $V(\mathfrak{h})=V(\mathfrak{h}^1)\cup V(\mathfrak{h}^2);$

\item $\big\{V_s(\mathfrak{h}^1),V_s(\mathfrak{h}^2)\big\}$ is a partition of $V_s(\mathfrak{h});$

\item if $x \in V_s(\mathfrak{h}^i),~f \in V_f(\mathfrak{h})$ and $x\sim f$, then $f\in V_f(\mathfrak{h}^i);$

\item if $x \in V_s(\mathfrak{h}^1)$ and $y \in V_s(\mathfrak{h}^2)$, then $x$ and $y$ have at most one common fat neighbor, and they have one if and only if they are adjacent.
\end{enumerate}
\end{lem}

\begin{lem}(\cite[Lemma $3.4$]{Koolen})\label{connected}
A Hoffman graph $\mathfrak{h}$ is indecomposable if and only if its special graph $\mathcal{S}(\mathfrak{h})$ is connected.
\end{lem}

\begin{lem}(\cite[Lemma $2.12$]{Koolen})
Suppose Hoffman graph $\mathfrak{h}$ is the direct sum of Hoffman graphs $\mathfrak{h}^1$ and $\mathfrak{h}^2$, that is $\mathfrak{h} = \mathfrak{h}^1 \oplus \mathfrak{h}^2$, then $\lambda_{\min}=\min\{\lambda_{\min}(\mathfrak{h}^1), \lambda_{\min}(\mathfrak{h}^2)\}$.
\end{lem}

For an integer $t\geq1$, let $\mathfrak{h}^{(t)}$ be the Hoffman graph with one slim vertex and $t$ fat neighbors. The proof of \cite[Lemma $3.5$]{Koolen} gives us the following lemma.
\begin{lem}\label{maximalweight}
Let $t$ be a positive integer. If the Hoffman graph $\mathfrak{h}$ with $\lambda_{\min}(\mathfrak{h})\geq-t$ contains $\mathfrak{h}^{(t)}$ as an induced Hoffman subgraph, then $\lambda_{\min}(\mathfrak{h})=-t$ and $\mathfrak{h}=\mathfrak{h}^{(t)}\oplus\mathfrak{h}^\prime$ for some induced Hoffman subgraph $\mathfrak{h}^\prime$ of $\mathfrak{h}$. In particular, if $\mathfrak{h}$ is indecomposable, then $\mathfrak{h}=\mathfrak{h}^{(t)}$.
\end{lem}

\begin{de} Let $\mu$ be a real number with $\mu\leq -1$ and let $\mathfrak{h}$ be a Hoffman graph with $\lambda_{\min}(\mathfrak{h})\geq \mu$. Then $\mathfrak{h}$ is called \emph{$\mu$-saturated} if no fat vertex can be attached to $\mathfrak{h}$ in such a way that the resulting Hoffman graph has smallest eigenvalue at least $\mu$.
\end{de}

\begin{de}\label{irreducible} Let $\mu$ be a real number with $\mu\leq -1$ and let $\mathfrak{h}$ be a Hoffman graph with $\lambda_{\min}(\mathfrak{h})\geq \mu$. The Hoffman graph $\mathfrak{h}$ is said to be \emph{$\mu$-reducible} if there exists a Hoffman graph $\widetilde{\mathfrak{h}}$ containing $\mathfrak{h}$ as an induced Hoffman subgraph such that there is a decomposition $\{\widetilde{\mathfrak{h}}_{i}\}_{i=1}^{2}$ of $\widetilde{\mathfrak{h}}$ with $\lambda_{\min}(\widetilde{\mathfrak{h}}_{i})\geq \mu$ and $V_{s}(\widetilde{\mathfrak{h}}_{i})\cap V_{s}(\mathfrak{h})\neq \emptyset\ (i=1,\ 2)$. We say that $\mathfrak{h}$ is \emph{$\mu$-irreducible} if $\lambda_{\min}(\mathfrak{h})\geq \mu$ and $\mathfrak{h}$ is not $\mu$-reducible. A Hoffman graph $\mathfrak{h}$ is said to be \emph{reducible} if $\mathfrak{h}$ is $\lambda_{\min}(\mathfrak{h})$-reducible. We say $\mathfrak{h}$ is \emph{irreducible} if $\mathfrak{h}$ is not reducible. \end{de}

In particular, an irreducible Hoffman graph is indecomposable. But the converse is not true in general.

\subsection{Some results on an integrally representable Hoffman graph $\mathfrak{h}$ of norm $3$}
Let the Hoffman graph $\mathfrak{h}$ be integrally representable of norm $3$ and $\phi: V(\mathfrak{h}) \rightarrow \mathbb{Z}^{n}$ be an integral representation of norm $3$. From Lemma \ref{reprerelation}, we may assume that $\phi$ is a mapping from $V(\mathfrak{h})$ to $\mathbb{Z}^n\bigoplus\mathbb{Z}^{|V_f(\mathfrak{h})|}$, where its composition with the projection $\mathbb{Z}^n\bigoplus\mathbb{Z}^{|V_f(\mathfrak{h})|}\to\mathbb{Z}^n$ gives an integral reduced representation $\psi:V_s(\mathfrak{h})\to\mathbb{Z}^n$. Therefore, in this paper, when a  Hoffman graph $\mathfrak{h}$ is integrally representable of norm $3$, we always write
$$\phi(x)=\psi(x)+\sum_{f\in N_{\mathfrak{h}}^f(x)}{\bf e}_f, \text{ and }$$
$$\psi(x)=\sum_{j=1}^n\psi(x)_j{\bf e}_j,~\psi(x)_j\in\{0,\pm1\}.$$

By Definition \ref{reducedrep}, we have
$$\big|\big\{j\mid j\in\{1,2,\ldots,n\},\psi(x)_j\in\{1,-1\}\big\}\big|=3-|N_{\mathfrak{h}}^f(x)|.$$

Let $x_1$ and $x_2$ be two distinct slim vertices of $\mathfrak{h}$. If $N_{\mathfrak{h}}^f(x_1)\neq\emptyset$ and $N_{\mathfrak{h}}^f(x_2)\neq\emptyset$, then, from \cite[Lemma $3.6$]{Koolen}, we have $(\psi(x_1),\psi(x_2))\in\{0,\pm1\}$ and if $N_{\mathfrak{h}}^f(x_1)=\emptyset$ or $N_{\mathfrak{h}}^f(x_2)=\emptyset$, then we have $(\psi(x_1),\psi(x_2))\in\{0,1\}$ by Definition \ref{reducedrep}.

The following three results will be important for this paper.

\begin{lem}(\cite[Lemma $4.2$ and Claim $4.4$]{Koolen})\label{extend}
Let $\mathfrak{h}$ be a fat, $(-3)$-saturated Hoffman graph with an integral representation $\psi$ of norm $3$ and let $I=\big\{i\mid i\in\{1,\ldots,n\}$ and there exists $x\in V_s(\mathfrak{h})$ such that $\psi(x)_i\neq 0\big\}$. Then the following holds:
\begin{enumerate}
\item For each $i\in I$, there exist $x_1,x_2\in V_s(\mathfrak{h})$ such that $\psi(x_1)_i=-\psi(x_2)_i=1$.
\item If there exists $x_1,x_2\in V_s(\mathfrak{h})$ and $i\in I$ such that $\psi(x_1)_i=-\psi(x_2)_i=1$, then $x_1$ and $x_2$ have distance at most $2$ in $S^-(\mathfrak{h})$.
\end{enumerate}
\end{lem}

\begin{lem}\cite{KLY}\label{contained}
Let $\mathfrak{h}$ be an indecomposable Hoffman graph with an integral representation of norm $3$, then there exists a Hoffman graph $\widetilde{\mathfrak{h}}$ satisfying the following properties:
\begin{enumerate}
\item $\widetilde{\mathfrak{h}}$ is fat, $(-3)$-saturated and integrally representable of norm $3$;
\item $\widetilde{\mathfrak{h}}$ has $\mathfrak{h}$ as an induced Hoffman subgraph;
\item $\widetilde{\mathfrak{h}}$ has the same slim graph as $\mathfrak{h}$.
\end{enumerate}
\end{lem}

\begin{thm}(\cite[Theorem $3.7$ and Theorem $4.9$]{Koolen})\label{lattice}
Let $\mathfrak{h}$ be a fat and indecomposable Hoffman graph with $\lambda_{\min}(\mathfrak{h})\geq-3$. Then the following holds:
\begin{enumerate}
\item If $\mathfrak{h}$ is $(-3)$-saturated and integrally representable of norm $3$, then its special $(-)$-graph $S^{-}(\mathfrak{h})$ is isomorphic to $A_{m},D_{m},\widetilde{A}_{m-1}$ or $\widetilde{D}_{m-1}$, where $m=|V_{s}(\mathfrak{h})|$.
\item If $\mathfrak{h}$ contains a slim vertex with at least two fat neighbors, then $\mathfrak{h}$ is integrally representable of norm $3$.
\end{enumerate}
\end{thm}

\section{Tree-like Hoffman graphs}\label{construction}
\subsection{Stripped Hoffman graphs}
Let $\mathfrak{t}_{1}$ and $\mathfrak{t}_{2}$ be two tree-like Hoffman graphs, then the Hoffman graph $\mathfrak{t}:= \mathfrak{t}_{1}\oplus\mathfrak{t}_{2}$ is not tree-like. So, in order to construct tree-like Hoffman graphs from smaller ones, we need to remove some fat vertices from the direct sum. This leads us to the definition of the \emph{stripped Hoffman graph}.
Let $\mathfrak{H}=\{\mathfrak{h}_i\mid i=1,\ldots,r\}$ be a finite family of Hoffman graphs. We define the stripped Hoffman graph $h_s(\mathfrak{H})$ by

$$h_s(\mathfrak{H})=\oplus_{i=1}^{r}\mathfrak{h}_i-\mathop\cup\limits_{i<j}V_f(\mathfrak{h}_i)\cap V_f(\mathfrak{h}_j).$$

\begin{lem}\label{fsum}
Let $\mathfrak{H}=\{\mathfrak{h}_i\mid i=1,\ldots,r\}$ be a family of connected Hoffman graphs. Then the Hoffman graph $h_s(\mathfrak{H})=h_s(\mathfrak{h}_1,\ldots,\mathfrak{h}_r)$ is tree-like if and only if the following conditions are satisfied:
\begin{enumerate}
\item $h_s(\mathfrak{H})$ is connected;
\item $\mathfrak{h}_{i}$ is tree-like for $i=1,\ldots,r$;
\item $|V_{f}(\mathfrak{h}_{i})\cap V_{f}(\mathfrak{h}_{j})|\leq1$ for any $i\neq j$;
\item $V_{f}(\mathfrak{h}_{i_1})\cap V_{f}(\mathfrak{h}_{i_2})\cap V_{f}(\mathfrak{h}_{i_3})=\emptyset$ for all $1\leq i_1< i_2<i_3\leq r$;
\item if a fat vertex $f\in V_{f}(\mathfrak{h}_{i})\cap V_{f}(\mathfrak{h}_{j})$ where $i\neq j$, then $f$ is a leaf for at lest one of them.
\end{enumerate}
\end{lem}

\begin{proof} It is straightforward to check that the conditions (i)-(v) are sufficient and necessary to ensure that $h_s(\mathfrak{H})$ is tree-like.
\end{proof}

\begin{re}\label{condition}
\begin{enumerate}
\item A tree-like Hoffman graph is indecomposable.
\item Let $\mathfrak{H}=\{\mathfrak{h}_i\mid i=1,\ldots,r\}$ be a family of Hoffman graphs. If $V_{f}(\mathfrak{h}_{i_1})\cap V_{f}(\mathfrak{h}_{i_2})\cap V_{f}(\mathfrak{h}_{r})=\emptyset$ for all $1\leq i_1< i_2\leq r-1$, then $h_s(\mathfrak{h}_1,\ldots,\mathfrak{h}_r)=h_s(h_s(\mathfrak{h}_1,\ldots,\mathfrak{h}_{r-1}),\mathfrak{h}_r)$.
\item If a tree-like Hoffman graph $\mathfrak{t}$ is an induced Hoffman subgraph of the Hoffman graph $\mathfrak{h}_1\oplus\mathfrak{h}_2$ where $V_s(\mathfrak{h}_1\oplus\mathfrak{h}_2)=V_s(\mathfrak{t})$, then $\mathfrak{t}$ is an induced Hoffman subgraph of $h_s(\mathfrak{h}_1,\mathfrak{h}_2)$.
\end{enumerate}
\end{re}

\subsection{Some results on tree-like Hoffman graphs}
In this subsection, we will give several results about tree-like Hoffman graphs with smallest eigenvalue at least $-3$.
\begin{pro}
Let $\mathfrak{t}=(T,l)$ be a tree-like Hoffman graph with $\lambda_{\min}(\mathfrak{t})\geq-3$. Then the special graph of $\mathfrak{t}$ is an edge-signed tree or $\mathfrak{t}\cong$\raisebox{-0.3ex}{\includegraphics[scale=0.13]{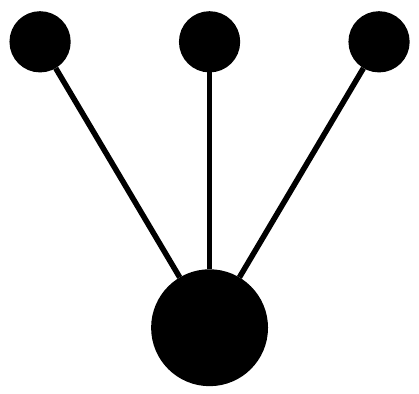}}.
\end{pro}

\begin{proof} Note that the Hoffman graph \raisebox{-0.3ex}{\includegraphics[scale=0.13]{photo3}} is $(-3)$-saturated (see \cite[Table $1$]{KLY}). So, if  \raisebox{-0.3ex}{\includegraphics[scale=0.13]{photo3}} is a proper induced Hoffman subgraph of $\mathfrak{t}$, then $\mathfrak{t}$ contains \raisebox{-0.3ex}{\includegraphics[scale=0.13]{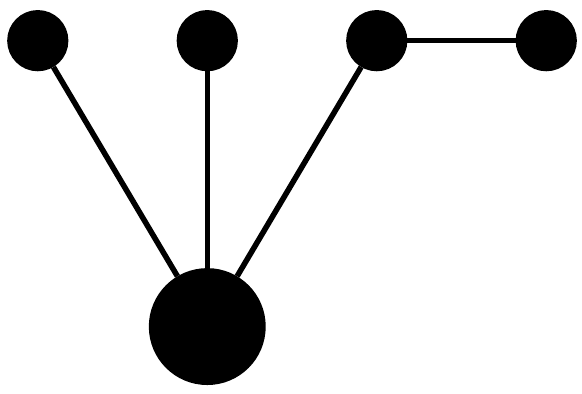}} or \raisebox{-0.3ex}{\includegraphics[scale=0.13]{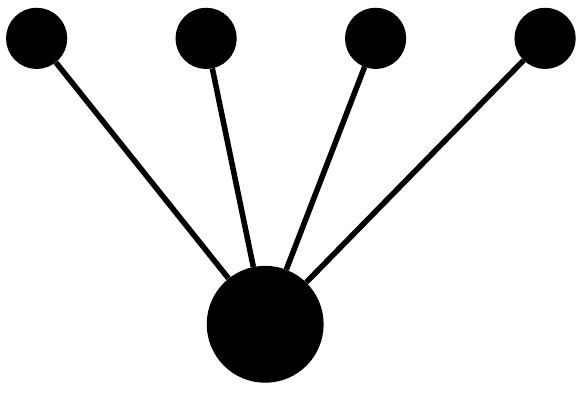}} as an induced Hoffman subgraph, because $\mathfrak{t}$ is tree-like. But $\lambda_{\min}(\raisebox{-0.4ex}{\includegraphics[scale=0.13]{section4induced1}})<-3$ and $\lambda_{\min}(\raisebox{-0.4ex}{\includegraphics[scale=0.13]{section4induced2}})<-3$. This shows that if $\mathfrak{t}\not\cong$ \raisebox{-0.3ex}{\includegraphics[scale=0.13]{photo3}}, then every fat vertex of $\mathfrak{t}$ has at most two slim neighbors (which are not adjacent in $\mathfrak{t}$). Now the special graph of $\mathfrak{t}$ is obtained by replacing \raisebox{-0.5ex}{\includegraphics[scale=0.8]{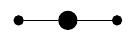}} by \raisebox{-0.5ex}{\includegraphics[scale=0.8]{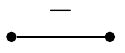}}, and removing the fat leaves of $\mathfrak{t}$. This shows the statement.
\end{proof}

\begin{lem}\label{edgesignedtree}
 Let $\mathcal{T}=(T,\text{sgn})$ be an edge-signed tree. Then $(T,\text{sgn})$ is switching equivalent to $(T,-)$. Therefore, for the signed adjacency matrix $B(\mathcal{T})$ of $\mathcal{T}$, there exists a real diagonal matrix $D$ such that $DB(\mathcal{T})D\leq0$, where $D^{-1}=D$.
 \end{lem}
\begin{proof}
This follows by an easy induction on the number of vertices.
\end{proof}

\begin{lem}\label{-similar}
Let $\mathfrak{t}=(T,l)$ be a tree-like Hoffman graph with $\lambda_{\min}(\mathfrak{t})\geq-3$. Then the following holds:
\begin{enumerate}
\item There exists a real diagonal matrix $D$ such that $DSp(\mathfrak{t})D\leq0$, where $Sp(\mathfrak{t})$ is the special matrix of $\mathfrak{t}$ and $D^{-1}=D$.
\item The multiplicity of $\lambda_{\min}(\mathfrak{t})$ is equal to 1.
\end{enumerate}
\end{lem}
\begin{proof}
If $\mathfrak{t}=$\raisebox{-0.3ex}{\includegraphics[scale=0.13]{photo3}}, $Sp(\mathfrak{t})=\begin{pmatrix}
-1 & -1 & -1 \\
-1 & -1 & -1 \\
-1 & -1 & -1
\end{pmatrix}$. If $\mathfrak{t}\neq$\raisebox{-0.3ex}{\includegraphics[scale=0.13]{photo3}}, its special graph $\mathcal{S}(\mathfrak{t})$ is an edge-signed tree $(T,\text{sgn})$ and for its special matrix $Sp(\mathfrak{t})$, there exists a real diagonal matrix $D$ such that $DSp(\mathfrak{t})D\leq0$ where $D^{-1}=D$ by Lemma \ref{edgesignedtree}, as each off-diagonal entry of $Sp(\mathfrak{t})$ has the same sign of the corresponding entry of the sign adjacency matrix of $(T,\text{sgn})$. Hence (i) holds. Since $-Sp(\mathfrak{t})$ is similar to a non-negative irreducible matrix, (ii) follows by Theorem \ref{perronfrobenius} (ii).
\end{proof}

\begin{pro}\label{smalleigenvalue}
Let $\mathfrak{t}_1=(T_1,l_1)$ and $\mathfrak{t}_2=(T_2,l_2)$ be tree-like Hoffman graphs with $\lambda_{\min}(\mathfrak{t}_{1}), \lambda_{\min}(\mathfrak{t}_{2})\geq-3$. Assume that the Hoffman graph $\mathfrak{t}=h_s(\mathfrak{t}_1,\mathfrak{t}_2)$ is tree-like. Then one of the following holds:
\begin{enumerate}
\item $(\lambda_{\min}(\mathfrak{t}_1)-\lambda_{\min}(\mathfrak{t}))(\lambda_{\min}(\mathfrak{t}_2)-\lambda_{\min}(\mathfrak{t}))<0$;
\item $\lambda_{\min}(\mathfrak{t}_1)= \lambda_{\min}(\mathfrak{t}_2)=\lambda_{\min}(\mathfrak{t})$.
\end{enumerate}
\end{pro}
\begin{proof}
Suppose $V_s(\mathfrak{t}_{1})=\{x_1,\ldots,x_p\}$ and $V_s(\mathfrak{t}_{2})=\{y_1,\ldots,y_q\}$. As $\mathfrak{t}$ is a tree-like Hoffman graph, we may assume, by Lemma \ref{fsum}, that $|V_f(\mathfrak{t}_1)\cap V_f(\mathfrak{t}_2)|=|\{f\}|=1$ and the fat vertex $f$ is a leaf of $\mathfrak{t}_1$ with $V_{\mathfrak{t}_1}^s(f)=\{x_p\}$ and $V_{\mathfrak{t}_2}^s(f)=\{y_1,\ldots,y_{q_1}\}$.

Since $\mathfrak{t}_1$ and $\mathfrak{t}_2$ are tree-like Hoffman graphs, there exist real diagonal matrices $D_1$ and $D_2$ such that $D_1^{-1}=D_1$, $D_2^{-1}=D_2$, $D_1Sp(\mathfrak{t}_{1})D_1\leq0$ and $D_2Sp(\mathfrak{t}_{2})D_2\leq0$ all hold. For $i=1,2$, let $\widetilde{S}_i=-D_i^{-1}Sp(\mathfrak{t_i})D_i$ and $\lambda_i=\lambda_{\max}(\widetilde{S}_i)=-\lambda_{\min}(\mathfrak{t}_i)$. By Theorem \ref{perronfrobenius}, there exist vectors ${\bf u}$ and ${\bf v}$ such that ${\bf u},{\bf v}>0$, $\widetilde{S}_1{\bf u}=\lambda_1{\bf u}$ and $\widetilde{S}_2{\bf v}=\lambda_2{\bf v}$. We may also assume ${\bf u}_{x_p}=\sum_{i=1}^{q_1}{\bf v}_{y_i}$.

Let $Sp(\mathfrak{t})=\begin{pmatrix}\begin{smallmatrix}S_{11}& S_{12}\\ S_{21} & S_{22}\end{smallmatrix}\end{pmatrix}$, where $S_{11}=Sp(\mathfrak{t})|_{V_s(\mathfrak{t}_1)}$, $S_{22}=Sp(\mathfrak{t})|_{V_s(\mathfrak{t}_2)}$, and $S_{12}=S_{21}^T=\begin{pmatrix}\begin{smallmatrix}&  &  \text{\large{0}} & &  &  \\1 & \cdots & 1  & \cdots & 0\end{smallmatrix}\end{pmatrix}$. Then $\begin{pmatrix}\begin{smallmatrix}D_1 &  \\
& D_2\end{smallmatrix}\end{pmatrix}Sp(\mathfrak{t})\begin{pmatrix}\begin{smallmatrix}D_1 &  \\& D_2 \end{smallmatrix}\end{pmatrix}\leq0$ or $\begin{pmatrix}\begin{smallmatrix}-D_1 &  \\
& D_2\end{smallmatrix}\end{pmatrix}Sp(\mathfrak{t})\begin{pmatrix}\begin{smallmatrix}-D_1 &  \\& D_2 \end{smallmatrix}\end{pmatrix}\leq0$. Without loss of generality, we may assume that we are in the first case. Let
$\widetilde{S}=-\begin{pmatrix}\begin{smallmatrix}D_1 &  \\
& D_2\end{smallmatrix}\end{pmatrix}Sp(\mathfrak{t})\begin{pmatrix}\begin{smallmatrix}D_1 &  \\& D_2 \end{smallmatrix}\end{pmatrix}$ and $\lambda_0=\lambda_{\max}(\widetilde{S})=-\lambda_{\min}(\mathfrak{t})$. Then we have $\widetilde{S}\begin{pmatrix}\begin{smallmatrix}{\bf u} \\ {\bf v} \end{smallmatrix}\end{pmatrix}=\begin{pmatrix}\begin{smallmatrix}\lambda_1{\bf u} \\ \lambda_2{\bf v}\end{smallmatrix}\end{pmatrix}$. First, we assume $\lambda_1\geq\lambda_2$. Then
$\lambda_1\begin{pmatrix}\begin{smallmatrix}{\bf u} \\ {\bf v} \end{smallmatrix}\end{pmatrix}\geq\widetilde{S}\begin{pmatrix}\begin{smallmatrix}{\bf u} \\{\bf v} \end{smallmatrix}\end{pmatrix}=\begin{pmatrix}\begin{smallmatrix}\lambda_1{\bf u} \\ \lambda_2{\bf v} \end{smallmatrix}\end{pmatrix}\geq \lambda_2\begin{pmatrix}\begin{smallmatrix}{\bf u} \\ {\bf v} \end{smallmatrix}\end{pmatrix}$. This implies $\lambda_1\geq \lambda_0\geq\lambda_2$ by Theorem \ref{perronfrobenius}. Moreover, if $\lambda_0=\lambda_1$ or $\lambda_0=\lambda_2$, then $\begin{pmatrix}\begin{smallmatrix}{\bf u} \\{\bf v} \end{smallmatrix}\end{pmatrix}$ is an eigenvector of $\widetilde{S}$ for $\lambda_0$ and $\lambda_1=\lambda_2=\lambda_0$. The case $\lambda_{1}\leq \lambda_{2}$ follows in a similar fashion. This completes the proof.
\end{proof}

In general, we have the following proposition from Proposition \ref{smalleigenvalue}.

\begin{pro}\label{gesmalleigenvalue}
Let $\mathfrak{T}=\{\mathfrak{t}_i\mid i=1,\ldots,r\}$ be a family of tree-like Hoffman graphs with $\lambda_{\min}(\mathfrak{t}_{i})\geq-3$ for $i=1,\ldots,r$. Assume that the Hoffman graph $\mathfrak{t}=h_s(\mathfrak{T})=h_s(\mathfrak{t}_1,\ldots,\mathfrak{t}_r)$ is tree-like. Then the following holds:
\begin{enumerate}
\item $\min\{\lambda_{\min}(\mathfrak{t}_i)\mid i=1,\ldots,r\}\leq\lambda_{\min}(\mathfrak{t})\leq \max\{\lambda_{\min}(\mathfrak{t}_i)\mid i=1,\ldots,r\}$;
\item $\lambda_{\min}(\mathfrak{t})=-3$ if and only if $\lambda_{\min}(\mathfrak{t}_i)=-3$ for all $i\in\{1,\cdots,r\}$.
\end{enumerate}
\end{pro}

\begin{pro}\label{degree3tree}
Let $\mathfrak{t}$ be a tree-like Hoffman graph with all internal vertices having valency $3$ and all its leaves fat. Then $\mathfrak{t}=h_s(\mathfrak{t}_1,\ldots,\mathfrak{t}_m)$, where $m=|V_s(\mathfrak{t})|$ and $\mathfrak{t}_i$ is isomorphic to \raisebox{-0.5ex}{\includegraphics[scale=0.13]{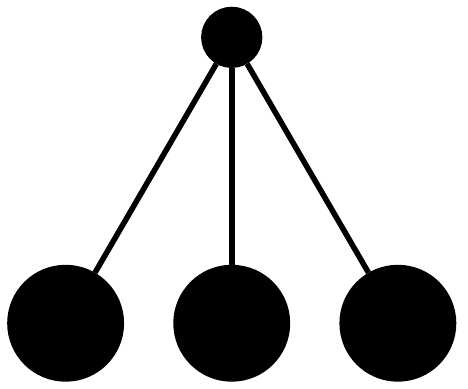}} for $i=1,\ldots,m$.

Moreover, $\mathfrak{t}$ is integrally representable of norm $3$ and $\lambda_{\min}(\mathfrak{t})=-3$.
\end{pro}
\begin{proof}
We proceed by induction on the number of slim vertices of $\mathfrak{t}$. The base case $m=1$ is trivial. Suppose $m\geq2$. Note that all its internal vertices are slim and there exists a slim vertex $x$ having only one slim neighbor $y$. Now attach a fat vertex to be exactly the fat neighbor of $x$ and $y$. Then the resulting Hoffman graph $\mathfrak{h}=\mathfrak{h}_1\oplus\mathfrak{h}_2$ and $\mathfrak{t}=h_s(\mathfrak{h}_1,\mathfrak{h}_2)$, where $\mathfrak{h}_1$ is the tree-like Hoffman graph satisfying the condition of this proposition and $\mathfrak{h}_2$ is isomorphic to \raisebox{-0.5ex}{\includegraphics[scale=0.13]{photo1}} with the slim vertex $x$. By the induction hypothesis, we have $\mathfrak{h}_1=h_s(\mathfrak{t}_1,\ldots,\mathfrak{t}_{m-1})$. Hence $\mathfrak{t}=h_s(h_s(\mathfrak{t}_1,\ldots,\mathfrak{t}_{m-1}),\mathfrak{h}_2)=h_s(\mathfrak{t}_1,\ldots,\mathfrak{t}_m)$ by Remark \ref{condition} (ii). Considering that the Hoffman graph  \raisebox{-0.5ex}{\includegraphics[scale=0.13]{photo1}} is integrally representable of norm $3$ with smallest eigenvalue $-3$, the proposition holds.
\end{proof}

Using Proposition \ref{degree3tree}, we obtain the following result.
 \begin{pro}\label{generaldegree3tree}
 Let $\mathfrak{t}$ be a tree-like Hoffman graph with all internal vertices have valency at most $3$ and any fat vertex is a leaf. Then $\mathfrak{t}$ is an induced Hoffman subgraph of a tree-like Hoffman graph $\tilde{\mathfrak{t}}=h_s(\mathfrak{t}_1,\ldots,\mathfrak{t}_m)$, where $m=|V_s(\mathfrak{t})|=|V_s(\tilde{\mathfrak{t}})|$ and $\mathfrak{t}_i$ is isomorphic to \raisebox{-0.5ex}{\includegraphics[scale=0.13]{photo1}} for $i=1,\ldots,m$. In particular, $\mathfrak{t}$ is integrally representable of norm $3$.
 \end{pro}

\subsection{A family of $(-3)$-irreducible tree-like Hoffman graphs}
In this subsection, we will introduce an infinite family $\mathfrak{F}$ of tree-like Hoffman graphs, which are integrally representable of norm $3$. In Section \ref{sectree}, we will use the tree-like Hoffman graphs in this family to construct the integrally representable tree-like Hoffman graphs of norm $3$.

\begin{de}\label{c}
Let $m\geq2$ be a positive integer. We define $\mathfrak{c}_m$ to be the tree-like Hoffman graph with slim vertex set $V_s(\mathfrak{c}_m)$ and fat vertex set $V_f(\mathfrak{c}_m)$ as follows:
$$V_s(\mathfrak{c}_m)=\{y_1,y_2,\ldots,y_m\},~V_f(\mathfrak{c}_m)=\{f_1,f_{1,2},f_m\}, \text{ where the adjacency relation is }$$
$$y_i\sim y_{i+2},~i=1,2,\ldots,m-2,$$
$$f_1\sim y_1,~f_{1,2}\sim y_1,~f_{1,2}\sim y_2,f_m\sim y_m.$$
\end{de}
Note that $\mathfrak{c}_2$ is the tree-like Hoffman graph \raisebox{-0.5ex}{\includegraphics[scale=0.12]{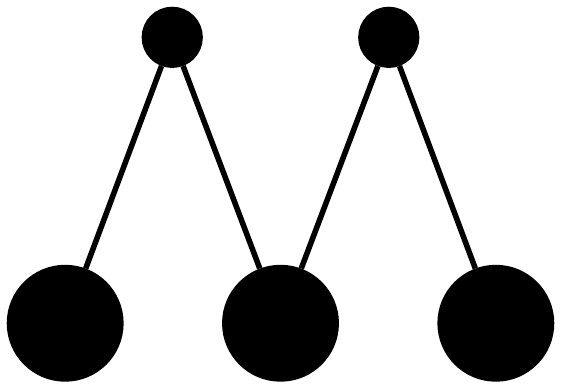}} and $\mathfrak{c}_3$ is the tree-like Hoffman graph \raisebox{-0.5ex}{\includegraphics[scale=0.12]{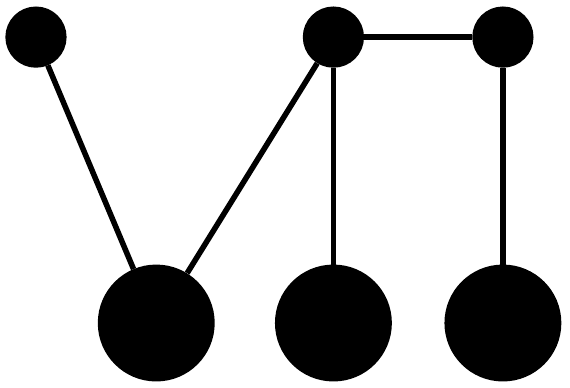}}.
\begin{de}\label{F}
We define the family $\mathfrak{F}$ of tree-like Hoffman graphs as the union of the families $\mathfrak{F}'$ and $\mathfrak{C}$, where $\mathfrak{F}'=\{\raisebox{-0.5ex}{\includegraphics[scale=0.12]{photo1},\includegraphics[scale=0.12]{photo3},
\includegraphics[scale=0.12]{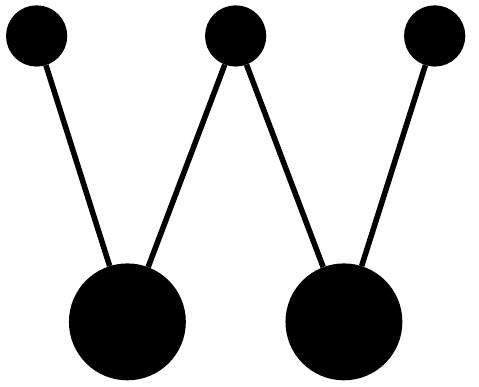}}\}$ and $\mathfrak{C}=\{\mathfrak{c}_i\mid i=2,3,\ldots\}$.
\end{de}

In Lemma \ref{lemF}, we will show that the tree-like Hoffman graphs in $\mathfrak{F}$ are $(-3)$-irreducible. Note that, by Definition \ref{irreducible}, if a Hoffman graph $\mathfrak{h}$ is $(-3)$-reducible, then there exists a way to attach fat vertices to $\mathfrak{h}$ such that the obtained Hoffman graph $\widetilde{\mathfrak{h}}$ is decomposable and satisfies $\lambda_{\min}(\widetilde{\mathfrak{h}})\geq-3$.

\begin{lem}\label{lemF}
Let $\mathfrak{t}$ be a tree-like Hoffman graph in $\mathfrak{F}$. Then $\lambda_{\min}(\mathfrak{t})=-3$. Moreover, $\mathfrak{t}$ is $(-3)$-irreducible and integrally representable of norm $3$.
\end{lem}
\begin{proof}
If $\mathfrak{t}\in\mathfrak{F}'$, the result is obvious. Now we assume that $\mathfrak{t}=\mathfrak{c}_m\in\mathfrak{C}$.

First, we prove that $\lambda_{\min}(\mathfrak{c}_m)=-3$. It is easy to obtain that $\lambda_{\min}(\mathfrak{c}_2)=-3$. If $m\geq3$, define the reduced representation $\psi_{\mathfrak{c}_m}$ of $\mathfrak{c}_m$ of norm $3$ as follows:
$$\psi_{\mathfrak{c}_m}(y_1)={\bf e}_1~,\psi_{\mathfrak{c}_m}(y_2)=-{\bf e}_1+{\bf e}_2~,$$
$$ \psi_{\mathfrak{c}_m}(y_i)={\bf e}_{i-2}+(-1)^{i-1}{\bf e}_{i-1}+{\bf e}_i,~i=3,\ldots,m-1,$$
$$\psi_{\mathfrak{c}_m}(y_m)={\bf e}_{m-2}+(-1)^{m-1}{\bf e}_{m-1}.$$
This shows that $\mathfrak{c}_m$ is integrally representable of norm $3$. Let $N$ be the matrix whose ${y_i}^{th}$ column is equal to $\psi_{\mathfrak{c}_m}(y_i)$, for $i=1,...,m$. Then $Sp(\mathfrak{c}_m)+3I=N^TN$ and $\lambda_{\min}(\mathfrak{c}_m)\geq-3$. Note that $N$ is an $(m-1)\times m$ matrix. This implies that the rank of the matrix $Sp(\mathfrak{c}_m)+3I$ is at most $m-1$ and $-3$ is an eigenvalue of $\mathfrak{c}_m$. Hence $\lambda_{\min}(\mathfrak{c}_m)=-3$.

Now we show that $\mathfrak{c}_m$ is $(-3)$-irreducible. Suppose this is not the case. Let $\tilde{\mathfrak{c}}_m$ be the decomposable Hoffman graph with $\lambda_{\min}(\tilde{\mathfrak{c}}_m)\geq-3$, which is obtained by attaching smallest number of fat vertices to $\mathfrak{c}_m$. Let $\tilde{\mathfrak{c}}_{m,1}$ be the indecomposable factor of $\tilde{\mathfrak{c}}_m$ with $y_1\in V_s(\tilde{\mathfrak{c}}_{m,1})$. Note that there is no fat vertex $f\in V_f(\tilde{\mathfrak{c}}_{m,1})-V_f(\mathfrak{c}_m)$ such that $f\in N_{\tilde{\mathfrak{c}}_{m,1}}^f(y_i,y_j)$ with $i-j=1(\mod 2)$, where $y_i,y_j\in V_s(\tilde{\mathfrak{c}}_{m,1})$. Let $m'=\min\{i\mid y_i$ has a fat neighbor in the set $V_f(\tilde{\mathfrak{c}}_{m,1})-V_f(\mathfrak{c}_m)\}$. It is easy to check that $2\leq m'< m$. Then $\tilde{\mathfrak{c}}_{m,1}$ has the tree-like Hoffman graph $\mathfrak{c}_{m'}(y_{m'+1})$, which is obtained by attaching slim vertex $y_{m'+1}$ to $\mathfrak{c}_{m'}$ as a slim neighbor of the vertex $y_{m'-1}$, as an induced Hoffman subgraph. By Lemma \ref{-similar} (i) and Theorem \ref{perronfrobenius} (v), we have that $\lambda_{\min}(\mathfrak{c}_{m'})=-3>\lambda_{\min}(\mathfrak{c}_{m'}(y_{m'+1}))$. Hence $-3>\lambda_{\min}(\mathfrak{c}_{m'}(y_{m'+1}))\geq\lambda_{\min}(\tilde{\mathfrak{c}}_{m,1})\geq\lambda_{\min}(\tilde{\mathfrak{c}}_m)=-3$. This gives a contradiction and the lemma holds.
\end{proof}

\begin{re}
The tree-like Hoffman graph $\mathfrak{c}_m$ has a unique integral representation of norm $3$, up to isomorphism.
\end{re}

\section{Integrally representable tree-like Hoffman graphs}\label{sectree}
In this section, we will study integrally representable tree-like Hoffman graphs of norm $3$ and in Corollary \ref{integraltree}, we will show that each of them is a stripped Hoffman graph of a finite family of tree-like Hoffman graphs. Now we start with the following lemma.

\begin{lem}\label{step1}
Let $\mathfrak{t}$ be an integrally representable tree-like Hoffman graph of norm $3$. Then there exists a family of tree-like Hoffman graphs $\mathfrak{T}=\{\mathfrak{t}_1$, \ldots, $\mathfrak{t}_{r(\mathfrak{t})}\}$ such that $\mathfrak{t}$ is an induced Hoffman subgraph of the Hoffman graph $\oplus_{i=1}^{r(\mathfrak{t})}\mathfrak{t}_i$, where
\begin{enumerate}
\item $\mathfrak{t}_i$ is integrally representable of norm $3$;
\item each Hoffman graph $\mathfrak{h}_i$ that is integrally representable of norm $3$ and contains $\mathfrak{t}_i$ as an induced Hoffman subgraph with $V_s(\mathfrak{h}_i)=V_s(\mathfrak{t}_i)$ is indecomposable;
\item the Hoffman graph $\oplus_{i=1}^{r(\mathfrak{t})}\mathfrak{t}_i$ has the same slim graph as $\mathfrak{t}$.
\end{enumerate}
\end{lem}

\begin{proof} The proof is given by induction on the number $m$ of slim vertices of $\mathfrak{t}$. If $m=1$, the result is obviously true. So we may assume that $m\geq2$ and there exists a decomposable and integrally representable Hoffman graph $\mathfrak{g}$ of norm $3$, that contains $\mathfrak{t}$ as an induced Hoffman subgraph with $V_s(\mathfrak{g})=V_s(\mathfrak{t})$, as otherwise $\mathfrak{T}=\{\mathfrak{t}\}$. Suppose $\mathfrak{g}=\oplus_{i=1}^{r}\mathfrak{g}_i$, where $\mathfrak{g}_i$ is the indecomposable factor of $\mathfrak{g}$ for $i=1,\ldots,r$ with $r\geq2$. Note that for each $i$, the slim graph of $\mathfrak{g}_i$ is a tree or forest. Let $\mathfrak{g}_i'$ be the induced subgraph of $\mathfrak{g}_i$, where $V_s(\mathfrak{g}_i')=V_s(\mathfrak{g}_i)$ and $V_f(\mathfrak{g}_i')=\{f\in V_f(\mathfrak{g}_i)\mid f\in V_f(\mathfrak{t})$ or there exists $j\neq i$ such that $f\in V_f(\mathfrak{g}_j)\}$. As a consequence, $\mathfrak{g}_i'$ is a disjoint union of tree-like Hoffman graphs, that is $\mathfrak{g}_{i}'=\cup_{j=1}^{r_i}\mathfrak{t}_{i,j}=\oplus_{j=1}^{r_i}\mathfrak{t}_{i,j}$, where $\mathfrak{t}_{i,j}$ is an integrally representable tree-like Hoffman graph of norm $3$. By the induction hypothesis, we have that $\mathfrak{t}_{i,j}$ is an induced Hoffman subgraph of the Hoffman graph $\oplus_{k=1}^{r(\mathfrak{t}_{i,j})}\mathfrak{t}_{i,j,k}$, where $\oplus_{k=1}^{r(\mathfrak{t}_{i,j})}\mathfrak{t}_{i,j,k}$ has the same slim graph as $\mathfrak{t}_{i,j}$ and $\mathfrak{t}_{i,j,k}$ satisfies (i) and (ii). Hence $\mathfrak{t}$ is an induced Hoffman subgraph of Hoffman graph $\oplus_{i=1}^{r}\oplus_{j=1}^{r_i}\oplus_{k=1}^{r(\mathfrak{t}_{i,j})}\mathfrak{t}_{i,j,k}$, which satisfies (i)-(iii) and this completes the proof.
\end{proof}

For the proof of the next theorem, we need to introduce the definition of the \emph{weighted graph}. A weighted graph is a pair $(G,w)$ of a graph $G$ and a weight function $w:V(G) \to \mathbb{Z}_{\geq 0}$. The \emph{weighted special $(-)$-graph} of a Hoffman graph $\frak{h}$ is the weighted graph $(S^-(\mathfrak{h}),w)$, where $w(x):=|\{f: f\ \text{ is a fat neighbor of } x \text{ in } \frak{h}\}|$.

\begin{thm}\label{step2}
Let $\mathfrak{T}$ be the family of non-isomorphic tree-like Hoffman graphs $\mathfrak{t}$ which satisfy the following properties:
 \begin{enumerate}
 \item $\mathfrak{t}$ is integrally representable of norm $3$;
 \item each Hoffman graph $\mathfrak{h}$ that is integrally representable of norm $3$ and contains $\mathfrak{t}$ as an induced Hoffman subgraph with $V_s(\mathfrak{h})=V_s(\mathfrak{t})$ is indecomposable.
 \end{enumerate}
Then the maximal tree-like Hoffman graphs in $\mathfrak{T}$ are isomorphic to one of the Hoffman graphs in $\mathfrak{F}$.
\end{thm}

\begin{proof}
Let $\mathfrak{t}$ be a maximal tree-like Hoffman graph in $\mathfrak{T}$. Lemma \ref{contained} shows that there exists a fat, $(-3)$-saturated and integrally representable Hoffman graph $\mathfrak{h}$ of norm $3$, whic is obtained by attaching fat vertices to $\mathfrak{t}$. The assumption of the theorem implies that $\mathfrak{h}$ is indecomposable. Let the graph $T$ be the slim graph of $\mathfrak{h}$ and $|V_s(\mathfrak{h})|=|V_s(\mathfrak{t})|=|V_s(T)|=|\{x_1,\ldots,x_m\}|=m$. Note that $T$ is a tree or forest.

If $m\leq3$, then Koolen et al. \cite{KLY} showed that $\mathfrak{h}$ is one of the Hoffman graphs in Figure \ref{ftype}.

\begin{figure}[H]
\ffigbox{
\begin{subfloatrow}
\ffigbox[1\FBwidth]{\caption*{$\mathfrak{h}_1$}}{\includegraphics[scale=0.8]{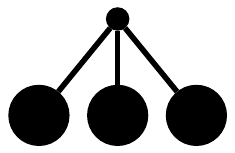}}\hspace{0.6cm}
\ffigbox[1\FBwidth]{\caption*{$\mathfrak{h}_2$}}{\includegraphics[scale=0.8]{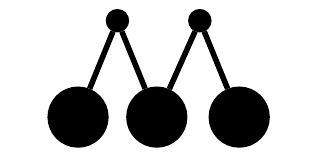}}\hspace{0.6cm}
\ffigbox[1\FBwidth]{\caption*{$\mathfrak{h}_3$}}{\includegraphics[scale=0.8]{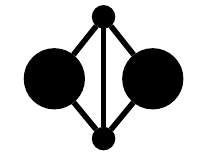}}\hspace{0.6cm}
\ffigbox[1\FBwidth]{\caption*{$\mathfrak{h}_4$}}{\includegraphics[scale=0.8]{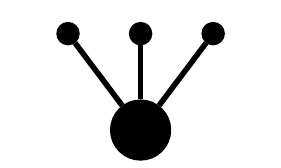}}
\end{subfloatrow}\vspace{0.3cm}

\begin{subfloatrow}
\ffigbox[1\FBwidth]{\caption*{$\mathfrak{h}_5$}}{\includegraphics[scale=0.8]{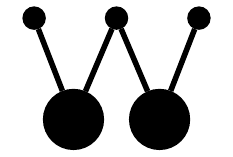}}\hspace{0.6cm}
\ffigbox[1\FBwidth]{\caption*{$\mathfrak{h}_6$}}{\includegraphics[scale=0.8]{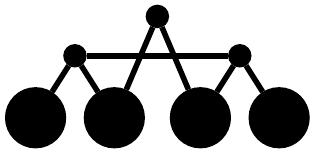}}\hspace{0.6cm}
\ffigbox[1\FBwidth]{\caption*{$\mathfrak{h}_7$}}{\includegraphics[scale=0.8]{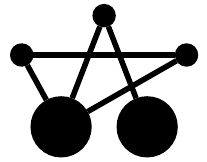}}\hspace{0.6cm}
\ffigbox[1\FBwidth]{\caption*{$\mathfrak{h}_8$}}{\includegraphics[scale=0.8]{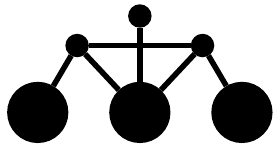}}
\end{subfloatrow}
}
{\caption{}\label{ftype}}
\end{figure}
Since the tree-like Hoffman graph $\mathfrak{t}$ is an induced Hoffman subgraph of $\mathfrak{h}_i$ for some $i\in\{1,\ldots,8\}$, it is easy to check that $\mathfrak{t}$ is exactly isomorphic to one of the Hoffman graphs \raisebox{-0.5ex}{\includegraphics[scale=0.12]{photo1},\includegraphics[scale=0.12]{photo2},\includegraphics[scale=0.12]{photo3},
\includegraphics[scale=0.12]{photo4},\includegraphics[scale=0.12]{photo5}} and the theorem holds.

Now we consider the case $m\geq4$. We will proceed with a sequence of claims, which we will use later in the proof. Denote by $S^-(\mathfrak{h})$ (resp. $(S^-(\mathfrak{h}),w)$) the special $(-)$-graph (resp. weighted special $(-)$-graph) of $\mathfrak{h}$. Denote by $\phi$ (resp. $\psi$) an integral representation (resp. reduced representation) of $\mathfrak{h}$ of norm $3$. If there exists a vertex $x_i$ such that $w(x_i)=3$ in $(S^-(\mathfrak{h}),w)$, then $\mathfrak{h}$ is isomorphic to the Hoffman graph \raisebox{-0.3ex}{\includegraphics[scale=0.12]{photo1}} by Lemma \ref{maximalweight} with $m=1$. Hence we have $w(x_i)\in\{1,2\}$ for any vertex $x_i$ of $S^-(\mathfrak{h})$.

\begin{claim}\label{pcommonfatneighbor}
Suppose that $y_1$ and $y_2$ are two distinct slim vertices of $\mathfrak{h}$, then $|N^f_{\mathfrak{h}}(y_1,y_2)|\leq1$.
\end{claim}
\begin{proof}[(Proof of Claim \ref{pcommonfatneighbor})] If $|N^f_{\mathfrak{h}}(y_1,y_2)|\geq2$, then $(\psi(y_1),\psi(y_2))=-1$ and $y_1\sim y_2$ in $\mathfrak{h}$ for $(\psi(y_1),\psi(y_2))\in\{0,\pm1\}$. From \cite{KLY}, we find that $\mathfrak{h}$ is exactly the Hoffman graph \raisebox{-0.8ex}{\includegraphics[scale=0.8]{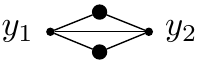}} and this contradicts the condition $m\geq4$.
\end{proof}

\begin{claim}\label{p22connected}
Suppose that the path $P=y_1y_2$ is an induced subgraph of $S^-(\mathfrak{h})$ with $w(y_1)=w(y_2)=2$. Then one of $y_1$ and $y_2$ is a leaf of $S^-(\mathfrak{h})$.
Moreover, if there exists a vertex $y_3$ adjacent to $y_2$ in $S^-(\mathfrak{h})$, then $y_3\sim y_1$ in $\mathfrak{h}$.
\end{claim}
\begin{proof}[(Proof of Claim \ref{p22connected})] By \cite{KLY}, if neither $y_1$ nor $y_2$ is a leaf of $S^-(\mathfrak{h})$, then $\mathfrak{h}$ is one of the Hoffman graphs \raisebox{-0.3ex}{\includegraphics[scale=0.12]{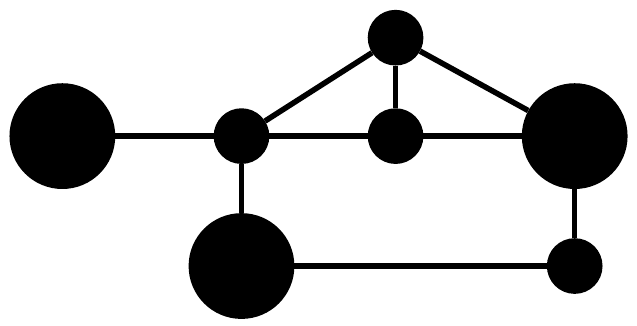}}, \raisebox{-0.3ex}{\includegraphics[scale=0.13]{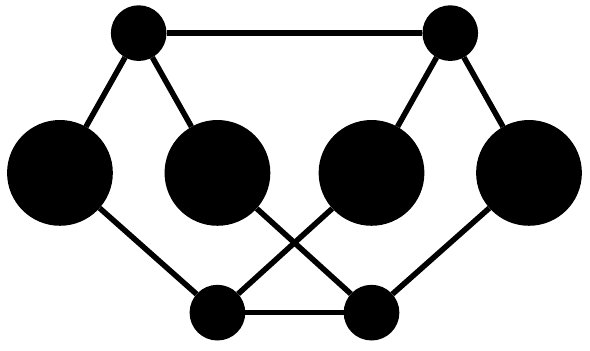}} and \raisebox{-0.3ex}{\includegraphics[scale=0.13]{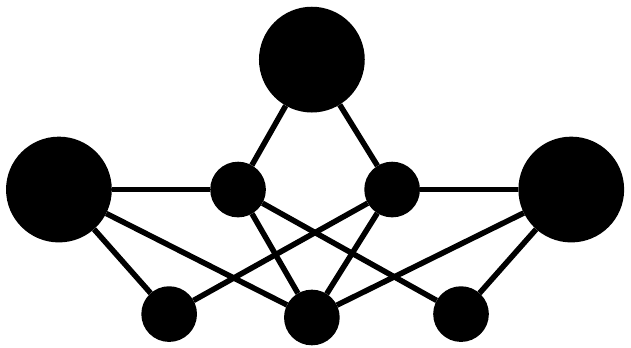}}. As the slim graph $T$ of $\mathfrak{h}$ has no cycle, $\mathfrak{h}$ is not the Hoffman graph \raisebox{-0.3ex}{\includegraphics[scale=0.13]{3leaf_1_hn}}. If $\mathfrak{h}$ is the Hoffman graph \raisebox{-0.3ex}{\includegraphics[scale=0.13]{p22connected_4_}}, then the tree-like Hoffman graph $\mathfrak{t}$ is exactly the Hoffman graph \raisebox{0.15ex}{\includegraphics[scale=0.13]{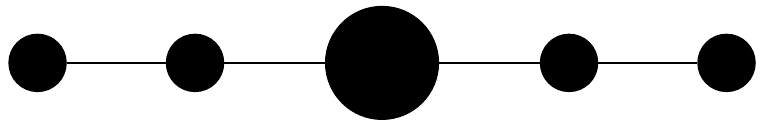}}. It is easy to see that $\mathfrak{t}$ is an induced subgraph of $\mathfrak{g}=$ \raisebox{-2.5ex}{\includegraphics[scale=0.8]{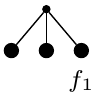}}$\oplus$\raisebox{-2.5ex}{\includegraphics[scale=0.8]{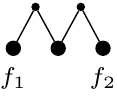}}$\oplus$\raisebox{-2.5ex}{\includegraphics[scale=0.8]{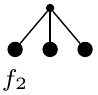}} and $\mathfrak{g}$ is integrally representable of norm $3$ and decomposable. This contradicts the assumption of the theorem. If $\mathfrak{h}$ is the Hoffman graph \raisebox{-0.3ex}{\includegraphics[scale=0.12]{p22connected_5_}}, then $\mathfrak{t}$ is exactly the Hoffman graph \raisebox{0.3ex}{\includegraphics[scale=0.13]{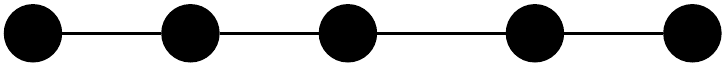}} and this gives a contradiction to the assumption of the theorem again by Proposition \ref{generaldegree3tree}. Without loss of generality, we may define $\psi(y_1)={\bf e}_1$ and $\psi(y_2)=-{\bf e}_1$, it is easy to see that if $y_3$ is adjacent to $y_2$ in $S^-(\mathfrak{h})$, then $(\psi(y_3),{\bf e_1})=1$ and $y_3\sim y_1$ in $\mathfrak{h}$.
\end{proof}

\begin{claim}\label{p22ends}
Suppose that the path $P=y_1y_2\cdots y_p$ is an induced subgraph of $S^-(\mathfrak{h})$ with $w(y_i)=1$ for $i=2,3,\ldots,p-1$. Then the graph $\bar{P}$, the complement of $P$, is an induced subgraph of $T$. Hence $\bar{P}$ has no cycle and $p\leq4$. In particular, if $p=3$ or $4$, then $y_1\sim y_p$ in $\mathfrak{h}$.
\end{claim}
\begin{proof}[(Proof of Claim \ref{p22ends})] This follows by the fact that the unique fat neighbor of $y_2$ is the fat neighbor of all the other $y_i$'s.
\end{proof}

\begin{claim}\label{p11endsfatneigbhor}
Suppose that the path $P=y_1y_2y_3$ is an induced subgraph of $S^-(\mathfrak{h})$ with $w(y_1)=w(y_3)=1$ and $w(y_2)=2$. Then $N_{\mathfrak{h}}^f(y_1,y_3)=\emptyset.$
\end{claim}
\begin{proof}[(Proof of Claim \ref{p11endsfatneigbhor})] Otherwise, $|N_{\mathfrak{h}}^f(y_{1},y_{3})|=1$ and $(\psi(y_{1}),\psi(y_{3}))\leq0$. As $y_{1}\not\sim y_{3}$ in $S^-(\mathfrak{h})$, we have $(\psi(y_{1}),\psi(y_{3}))=0$ and $y_1\sim y_3$ in $\mathfrak{h}$. Now we may define $\psi(y_{1})={\bf e}_1+{\bf e}_2,~\psi(y_{2})=-{\bf e}_2,~\psi(y_{3})=-{\bf e}_1+{\bf e}_2,$ and $\phi(y_{1})={\bf e}_1+{\bf e}_2+{\bf e}_{f_1},~\phi(y_{2})=-{\bf e}_2+{\bf e}_{f_1}+{\bf e}_{f_2},~\phi(y_{3})=-{\bf e}_1+{\bf e}_2+{\bf e}_{f_1}.$
If there exists a vertex $y_i$ adjacent to $y_{1}$ or $y_{3}$ in $S^-(\mathfrak{h})$ where $i\neq 2$, then $N_{\mathfrak{h}}^f(y_i,y_{j})=\{f_1\}$ and $(\psi(y_i),\psi(y_{j}))\leq0$ for $j=1,2,3$. This is not possible. Hence there exists a vertex $y_k$ adjacent to $y_2$ in $S^-(\mathfrak{h})$ as $S^-(\mathfrak{h})$ is connected by Theorem \ref{lattice} (i) and $|V_s(\mathfrak{h})|=m\geq4$. But we have $y_k\not\sim y_{1}$ and $y_k\not\sim y_{3}$ in $S^-(\mathfrak{h})$, since $S^-(\mathfrak{h})$ is isomorphic to $A_m,~\widetilde{A}_{m-1},~D_m$ or $\widetilde{D}_{m-1}$. Hence $(\psi(y_k),\psi(y_2))<0$, $(\psi(y_k),\psi(y_1))\geq0$ and $(\psi(y_k),\psi(y_{3}))\geq0$. It is easy to check that $(\psi(y_k),\psi(y_{1}))=0$ or $(\psi(y_k),\psi(y_{3}))=0$ is not possible. Hence $(\psi(y_k),\psi(y_{1}))>0$ and $(\psi(y_k),\psi(y_{3}))>0$ and this implies that $y_k\sim y_{1}$ and $y_k\sim y_{3}$ in $\mathfrak{h}$ and $y_{1}y_{3}y_k$ is a cycle of $T$. This gives a contradiction.
\end{proof}

\begin{claim}\label{p11ends}
Suppose the path $P=y_{1}y_{2}y_{3}$ is an induced subgraph of $S^-(\mathfrak{h})$ with $w(y_{1})=w(y_{3})=1$ and $w(y_{2})=2$. Then one of the following holds:
\begin{itemize}
\item[$(a)$] $y_{1}\sim y_{3}$ in $\mathfrak{h}$;
\item[$(b)$] $y_{1}$ or $y_{3}$ is a leaf of $S^-(\mathfrak{h})$.
\end{itemize}
\end{claim}
\begin{proof}[(Proof of Claim \ref{p11endsfatneigbhor})] Suppose this is not the case. Then $y_{1}\not\sim y_{3}$ in $\mathfrak{h}$. We also have $y_{1}\not\sim y_{3}$ in $S^-(\mathfrak{h})$. So $(\psi(y_1),\psi(y_3))=0$ and we may define
$\psi(y_{1})={\bf e}_1+{\bf e}_2,~\psi(y_{2})=-{\bf e}_2,~\psi(y_{3})=-{\bf e}_1+{\bf e}_2.$ From Claim \ref{p11endsfatneigbhor}, we find $N_{\mathfrak{h}}^f(y_{1},y_{3})=\emptyset.$ Hence without loss of generality, define $\phi(y_{1})={\bf e}_1+{\bf e}_2+{\bf e}_{f_1},~\phi(y_{2})=-{\bf e}_2+{\bf e}_{f_1}+{\bf e}_{f_2},~\phi(y_{3})=-{\bf e}_1+{\bf e}_2+{\bf e}_{f_2}.$
As neither $y_{1}$ nor $y_{3}$ is a leaf of $S^-(\mathfrak{h})$, there exists a vertex $y_i'\sim y_{i}$ in $S^-(\mathfrak{h})$ for $i=1,3$. It follows that $(\phi(y_1'),-{\bf e}_1+{\bf e}_{f_1})=2$ and $(\phi(y_3'),{\bf e}_1+{\bf e}_{f_2})=2$. From Lemma \ref{extend} \rm{(ii)}, the vertices $y_1'$ and $y_{3}'$ have distance at most $2$ in $S^-(\mathfrak{h})$. If there exists a vertex $y_4$ adjacent to $y_1'$ and $y_{3}'$ in $S^-(\mathfrak{h})$, then $w(y_4)=2$ as otherwise $y_1'\sim y_3'$ by Claim \ref{p22ends} and $y_1'y_2y_3'$ is a cycle in $T$. But, if $w(y_4)=2$, we can not find $\psi(y_4)$ satisfying $(\psi(y_4),\psi(y_1'))=-1$ and $(\psi(y_4),\psi(y_3'))=-1$, as $(\psi(y_1'),{\bf e}_1)=-1$, $(\psi(y_3'),{\bf e}_1)=1$ and $(\psi(y_1'),\psi(y_3'))\geq0$. Hence $y_1'\sim y_3'$ in $S^-(\mathfrak{h})$ and $V_s(\mathfrak{h})=\{y_{1},y_{2},y_{3},y_1',y_3'\}$. If $w(y_1')\neq2$ in $(S^-(\mathfrak{h}),w)$, then $\psi(y_1')=-{\bf e}_1\pm {\bf e}_q$ for some $q\geq3$ and $(\psi(y_3'),{\bf e}_q)=0$. Now we can not find $y\in V_s(\mathfrak{h})$ such that $\psi(y)_q=-\psi(y_1')_q$. This contracts Lemma \ref{extend} \rm{(i)}. So $w(y_1')=2$. Similarly, $w(y_3')=2$. Then by Claim \ref{p22connected}, we obtain a contradiction.
\end{proof}

\begin{claim}\label{dm}
The graph \raisebox{-0.6ex}{\includegraphics[scale=0.8]{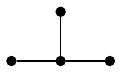}} can not be an induced subgraph of $S^-(\mathfrak{h})$.
\end{claim}
\begin{proof}[(Proof of Claim \ref{dm})] Assume \raisebox{-2.0ex}{\includegraphics[scale=0.8]{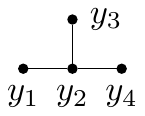}} is an induced subgraph of $S^-(\mathfrak{h})$ with $y_1$ and $y_3$ as leaves. From \cite{KLY}, we obtain that all leaves adjacent to $y_2$ in $S^-(\mathfrak{h})$ have weight $1$, otherwise $\mathfrak{h}$ is one of the Hoffman graphs \raisebox{-0.15ex}{\includegraphics[scale=0.12]{3leaf_1_hn}} and \raisebox{-0.15ex}{\includegraphics[scale=0.12]{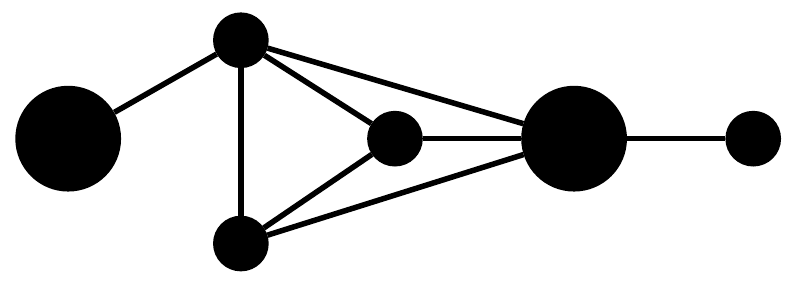}} and this is not possible as their slim graphs have a cycle. If $w(y_2)=1$, then $y_1y_3y_4$ is a cycle in $T$ by Claim \ref{p22ends}. Hence $w(y_2)=2$. If $w(y_4)=2$, then $y_4$ is a leaf in $\mathfrak{h}$ by Claim \ref{p22connected} and this is also not possible. So $w(y_4)=1$. Then for $y_1$, $y_2$ and $y_4$, there must be $2$ vertices share one common fat neighbor as $w(y_2)=2$ and it contradicts Claim \ref{p11endsfatneigbhor}.
\end{proof}

Claim \ref{dm} and Theorem \ref{lattice} imply that the special $(-)$-graph $S^-(\mathfrak{h})$ is isomorphic to either $A_m$ or $\widetilde{A}_{m-1}$. Let us first consider that $S^-(\mathfrak{h})\cong\widetilde{A}_{m-1}$, that is, $S^-(\mathfrak{h})$ is the cycle $x_1x_2\cdots x_mx_1$. For the path $P=x_{i-1}x_{i}x_{i+1}$ (where $i\mod m$) of $S^-(\mathfrak{h})$, we have $(w(x_{i-1}),w(x_i),w(x_{i+1}))\in\{(1,1,1),(1,1,2),(1,2,1),(2,1,1),(2,1,2)\}.$ By using Claim \ref{p22ends} and Claim \ref{p11ends}, we have $x_{i-1}\sim x_{i+1}$ in $\mathfrak{h}$ for all $i\in\{1,2,\ldots,m\}$. If $m$ is even, then $x_1x_3\cdots x_{m-1}$ is a cycle of $T$; if $m$ is odd, then $x_1x_3\cdots x_mx_2x_4\cdots x_{m-1}$ is a cycle of $T$. This shows that $S^-(\mathfrak{h})\not\cong\widetilde{A}_{m-1}$.

Hence $S^-(\mathfrak{h})\cong A_m$ and assume $S^-(\mathfrak{h})=x_1x_2\cdots x_m$ with $m\geq4$. By using Claim \ref{p22ends} and Claim \ref{p11ends} again, we find that there exists two paths $P_1$ and $P_2$ in the graph $T$, where
$$P_1=x_3x_5\cdots x_{2\lfloor\frac{m}{2}\rfloor-1},~P_2=x_2x_4\cdots x_{2\lfloor\frac{m-1}{2}\rfloor}.$$
As $T$ contains no cycle, both paths $P_1$ and $P_2$ are exactly the induced subgraphs of $T$. Moreover, there exists at most one edge between $P_1$ and $P_2$ in the graph $T$. Note that if there is a weighted path $(P,w|_P)$ in $(S^-(\mathfrak{h}),w)$, where $P=x_{i+1}x_{i+2}x_{i+3}x_{i+4}$ with $w(x_{i+1})=w(x_{i+4})=2$ and $w(x_{i+2})=w(x_{i+3})=1$, then \raisebox{-0.5ex}{\includegraphics[scale=0.8]{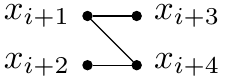}} is an induced subgraph of $T$ by Claim \ref{p22ends}. This implies that there is at most one weighted path isomorphic to $(P,w|_P)$ in $(S^-(\mathfrak{h}),w)$.

Now we focus on the possibilities for the tree-like Hoffman graph $\mathfrak{t}$. Note that for help of understanding, the numbers in the Hoffman graphs in Figure \ref{fproofcase1} and Figure \ref{fproofcase2} denote the number of its fat neighbors in the Hoffman graph $\mathfrak{h}$.

\begin{enumerate}
\item $w(x_1)=1$ or $w(x_m)=1$ in $(S^-(\mathfrak{h}),w)$.  Suppose $w(x_1)=1$, by \cite{KLY}, we find $w(x_2)=2$ and $w(x_3)=1$. Without loss of generality, we may define
$\psi(x_1)={\bf e}_1+{\bf e}_2,\psi(x_2)=-{\bf e}_2, \psi(x_3)=-{\bf e}_1+{\bf e}_2$ by Lemma \ref{extend} and $\phi(x_1)={\bf e}_1+{\bf e}_2+{\bf e}_{f_1}, \phi(x_2)=-{\bf e}_2+{\bf e}_{f_1}+{\bf e}_{f_2}, \phi(x_3)=-{\bf e}_1+{\bf e}_2+{\bf e}_{f_2}$ by Claim \ref{p11endsfatneigbhor}.

    If $w(x_4)=2$, then $\psi(x_4)={\bf e}_1$ and  $\phi(x_4)={\bf e}_1+{\bf e}_{f_2}+{\bf e}_{f_i}$ for some $i>3$. Now it is easy to check that there is no another vertex adjacent to $x_4$ in $S^-(\mathfrak{h})$ and $V_s(\mathfrak{h})=\{x_1,x_2,x_3,x_4\}$. The Hoffman graph $\mathfrak{t}$ is the Hoffman graph $\mathfrak{t}_{1,1}$ in Figure \ref{fproofcase1} and this is not possible.

    If $w(x_4)=1$, then $w(x_5)=2$ by Claim \ref{p22ends}. This also implies that $w(x_m)=2$, as otherwise we find two weighted paths isomorphic to $(P,w|_P)$ in $(S^-(\mathfrak{h}),w)$. If $w(x_{m-1})=1$, the Hoffman graph $\mathfrak{t}$ is the Hoffman graph $\mathfrak{t}_{1,2}$ in Figure \ref{fproofcase1}; if $w(x_{m-1})=2$, the Hoffman graph $\mathfrak{t}$ is the Hoffman graph $\mathfrak{t}_{1,3}$ in Figure \ref{fproofcase1}.
    \begin{figure}[H]
  \ffigbox{
  \begin{subfloatrow}
  \ffigbox[1\FBwidth]{\caption*{$\mathfrak{t}_{1,1}$}}{\includegraphics[scale=0.8]{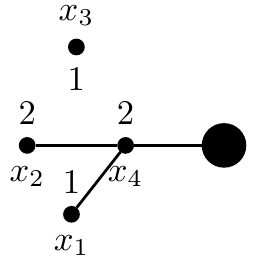}}\hspace{0.3cm}
  \ffigbox[1\FBwidth]{\caption*{$\mathfrak{t}_{1,2}$}}{\includegraphics[scale=0.8]{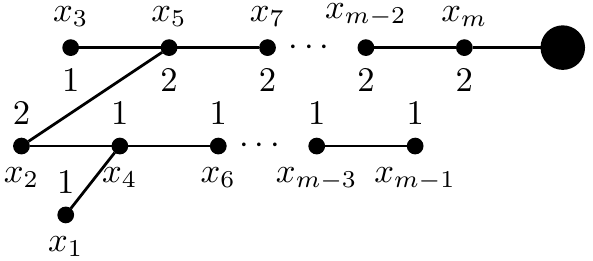}}\hspace{0.3cm}
  \ffigbox[1\FBwidth]{\caption*{$\mathfrak{t}_{1,3}$}}{\includegraphics[scale=0.8]{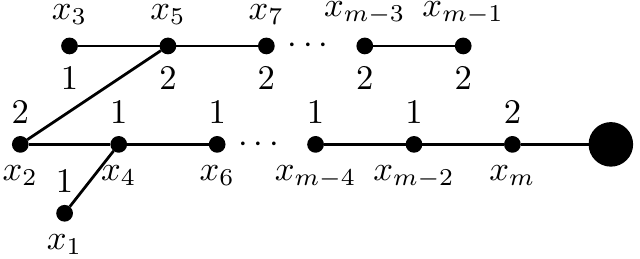}}\hspace{0.3cm}
  \end{subfloatrow}\vspace{0.3cm}
  }
  {\caption{}\label{fproofcase1}}
   \end{figure}
   By using Proposition \ref{generaldegree3tree}, we find that neither $\mathfrak{t}=\mathfrak{t}_{1,2}$ nor $\mathfrak{t}=\mathfrak{t}_{1,3}$, since $\mathfrak{t}$ satisfies the condition (ii) of Theorem \ref{step2}.
\item $w(x_1)=w(x_m)=2$. If $w(x_2)=w(x_{m-1})=1$ or there is a weighted path isomorphic to $(P,w|_P)$ in $(S^-(\mathfrak{h}),w)$, then $\mathfrak{t}$ is the tree-like Hoffman graph such that all its internal vertices have valency at most $3$ and all fat vertices are leaves. By using Proposition \ref{generaldegree3tree} again, we find that this is also not possible.

     Now we have $w(x_2)=2$ or $w(x_{m-1})=2$ and there is no weighted path isomorphic to $(P,w|_P)$ in $(S^-(\mathfrak{h}),w)$. Without loss of generality, we may assume $w(x_2)=2$. Then $\mathfrak{t}$ is one of the following in Figure \ref{fproofcase2} and $\mathfrak{t}$ is isomorphic to $\mathfrak{c}_m$.
        \begin{figure}[H]
        \ffigbox{
         \begin{subfloatrow}
         \ffigbox[1\FBwidth]{\caption*{$\mathfrak{t}_{2,1}$}}{\includegraphics[scale=0.7]{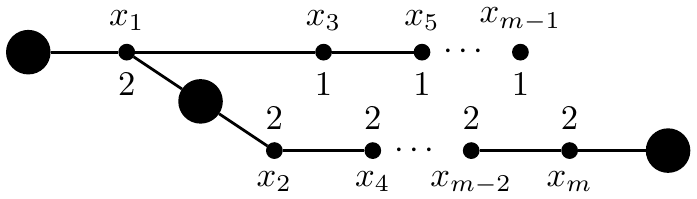}}\hspace{-0.3cm}
         \ffigbox[1\FBwidth]{\caption*{$\mathfrak{t}_{2,2}$}}{\includegraphics[scale=0.7]{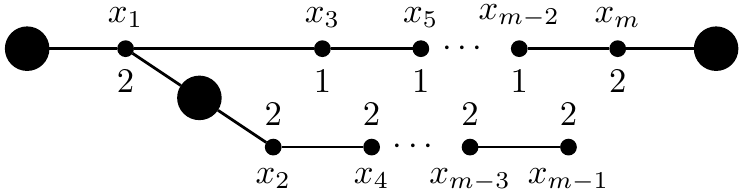}}\hspace{0cm}
         \ffigbox[1\FBwidth]{\caption*{$\mathfrak{t}_{2,3}$}}{\includegraphics[scale=0.7]{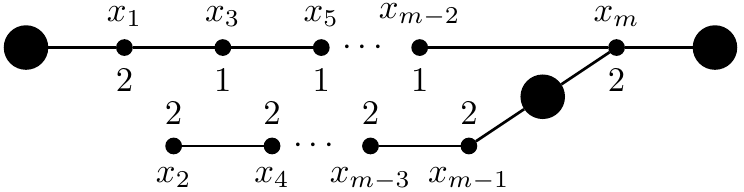}}\hspace{0cm}
         \end{subfloatrow}
          }
         {\caption{}\label{fproofcase2}}
         \end{figure}
\end{enumerate}
This completes the proof.
\end{proof}

\begin{cor}\label{integraltree}
Let $\mathfrak{t}$ be an integrally representable tree-like Hoffman graph of norm $3$ with $\lambda_{\min}(\mathfrak{t})=-3$, then $\mathfrak{t}=h_s(\mathfrak{t}_1,\ldots,\mathfrak{t}_{r(\mathfrak{t})})$, where $\mathfrak{t}_i$ is isomorphic to one of the Hoffman graphs in $\mathfrak{F}$, for $i=1,\ldots,r(\mathfrak{t})$.

Moreover, if $\mathfrak{t}$ is an integrally representable tree-like Hoffman graph of norm $3$, then $\mathfrak{t}$ is an induced Hoffman subgraph of an integrally representable tree-like Hoffman graph $\tilde{\mathfrak{t}}$ of norm $3$ with $\lambda_{\min}(\tilde{\mathfrak{t}})=-3$ and $V_s(\tilde{\mathfrak{t}})=V_s(\mathfrak{t})$.
\end{cor}
\begin{proof}
Proposition \ref{gesmalleigenvalue}, Lemma \ref{step1} and Theorem \ref{step2} show the result immediately.
\end{proof}

\begin{cor}\label{maintheoroem}
 Let $T$ be an integrally representable tree of norm $3$ with spectral radius $3$. Then $T=h_s(\mathfrak{t}_1,\ldots,\mathfrak{t}_{r(T)})$, where $\mathfrak{t}_i$ is isomorphic to one of the Hoffman graphs in $\mathfrak{F}$, for $i=1,\ldots,r(T)$.

Moreover, if $T$ is an integrally representable tree of norm $3$, then $T$ is an induced subgraph of an integrally representable tree of norm $3$ with spectral radius $3$.
\end{cor}
\begin{proof}
Let $T$ be an integral representable tree of norm $3$. If $\lambda_{\min}(T)=-3$, the result is obtained straightforward from Corollary \ref{integraltree}. If $\lambda_{\min}(T)>-3$, then $T$ is an induced subgraph of $h_s(\mathfrak{t}_1,\ldots,\mathfrak{t}_{r})$, where $h_s(\mathfrak{t}_1,\ldots,\mathfrak{t}_{r})$ has $T$ as slim graph and $\mathfrak{t}_i$ is isomorphic to one of the Hoffman graphs in $\mathfrak{F}$, for $i=1,\ldots,r$. Note that each fat vertex of $h_s(\mathfrak{t}_1,\ldots,\mathfrak{t}_{r})$ is a leaf. Assume $V_f(h_s(\mathfrak{t}_1,\ldots,\mathfrak{t}_{r}))=\{f_1,\ldots,f_p\}$ and define the Hoffman graph $h_s(\mathfrak{t}_1,\ldots,\mathfrak{t}_{r},\mathfrak{t}_1',\ldots,\mathfrak{t}_{p}')$, where $\mathfrak{t}_i'$ is isomorphic to  \raisebox{-0.5ex}{\includegraphics[scale=0.13]{photo3}} with the fat vertex $f_i$. Then $h_s(\mathfrak{t}_1,\ldots,\mathfrak{t}_{r},\mathfrak{t}_1',\ldots,\mathfrak{t}_{p}')$ with smallest eigenvalue $-3$ has no fat vertices and has $T$ as an induced subgraph. This completes the proof.
\end{proof}

\section{Seedlings}\label{secconclusion}
In this section, we introduce and discuss seedlings. They generalize the tree-like Hoffman graphs.
\begin{de}\label{seedling}
Let $\mu$ be a positive real number. A tree-like Hoffman graph $\mathfrak{t}$ with $\lambda_{\min}(\mathfrak{t})\geq -\mu$ is called a \emph{$\mu$-seedling} if the following conditions are satisfied:
\begin{enumerate}
\item $\mathfrak{t}$ is $(-\mu)$-irreducible;
\item $\mathfrak{t}$ is not a proper induced Hoffman subgraph of any $(-\mu)$-irreducible tree-like Hoffman graph.
\end{enumerate}
\end{de}

The only $1$-seedling is \raisebox{-0.5ex}{\includegraphics[scale=0.13]{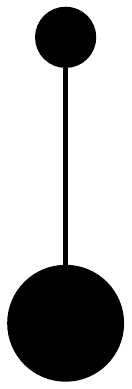}} and the $2$-seedlings are \raisebox{-0.5ex}{\includegraphics[scale=0.13]{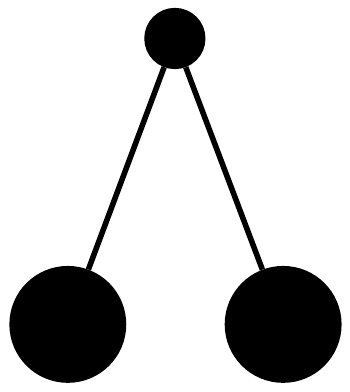}},
 \raisebox{-0.5ex}{\includegraphics[scale=0.13]{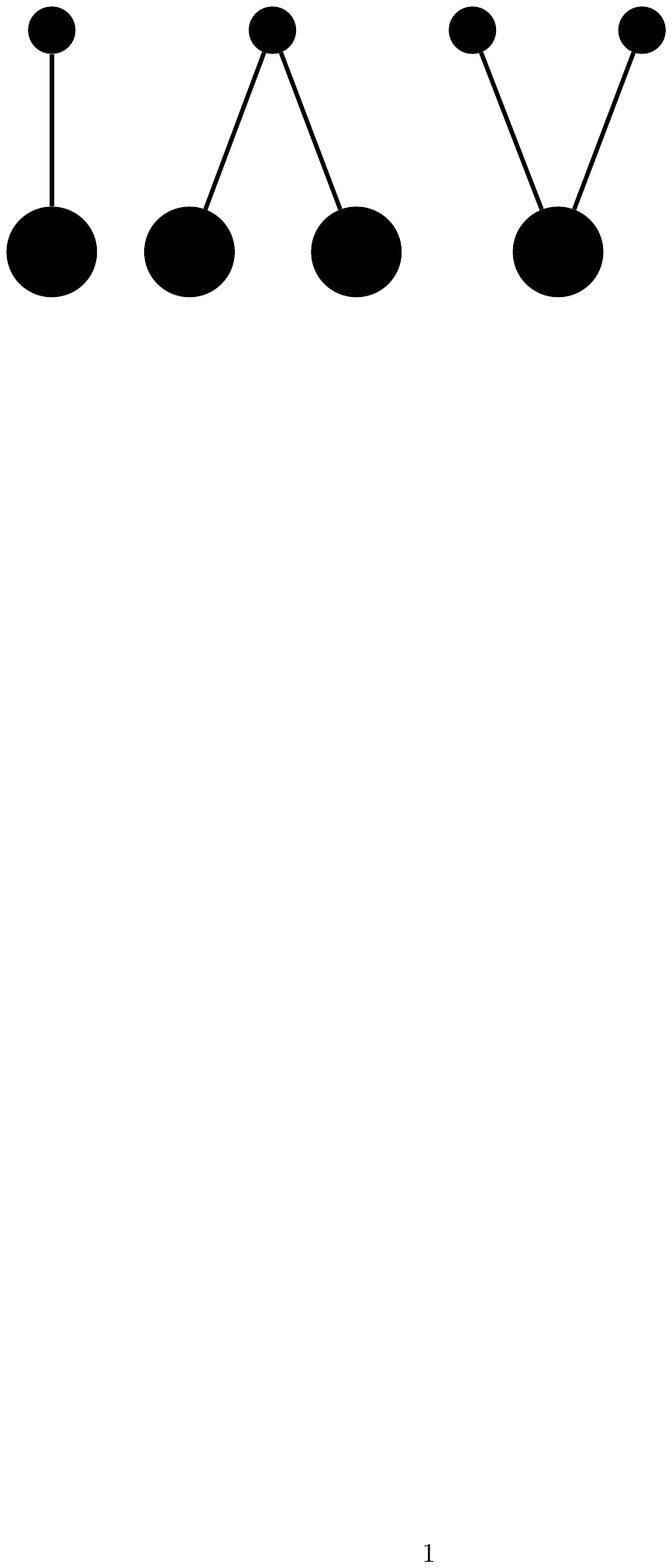}}, $\widetilde{E}_6$, $\widetilde{E}_7$ and $\widetilde{E}_8$. We think that it is important to classify the $3$-seedlings. This would help us to better understand integral lattices generated by norm $3$ vectors. For 3-seedlings, we have the following results.

 \begin{lem}\label{3seedling-3}
 Let $\mathfrak{t}$ be a $(-3)$-irreducible tree-like Hoffman graph with $\lambda_{\min}(\mathfrak{t})=-3$. Then $\mathfrak{t}$ is a $3$-seedling.
 \end{lem}
 \begin{proof}
This follows from Lemma \ref{-similar} (i) and Theorem \ref{perronfrobenius} (v) directly.
 \end{proof}

\begin{re}\label{F-3seedlings}
Each of the Hoffman graphs in $\mathfrak{F}$ is a $3$-seedling by Lemma \ref{3seedling-3} and Lemma \ref{lemF}. It follows, by Corollary \ref{integraltree}, that they are the only integrally representable $3$-seedlings of norm $3$.
\end{re}
Note that if you add a fat vertex to each of the slim vertices of a $\mu$-seedling, you obtain a $(\mu+1)$-seedling. This means that \raisebox{-0.5ex}{\includegraphics[scale=0.13]{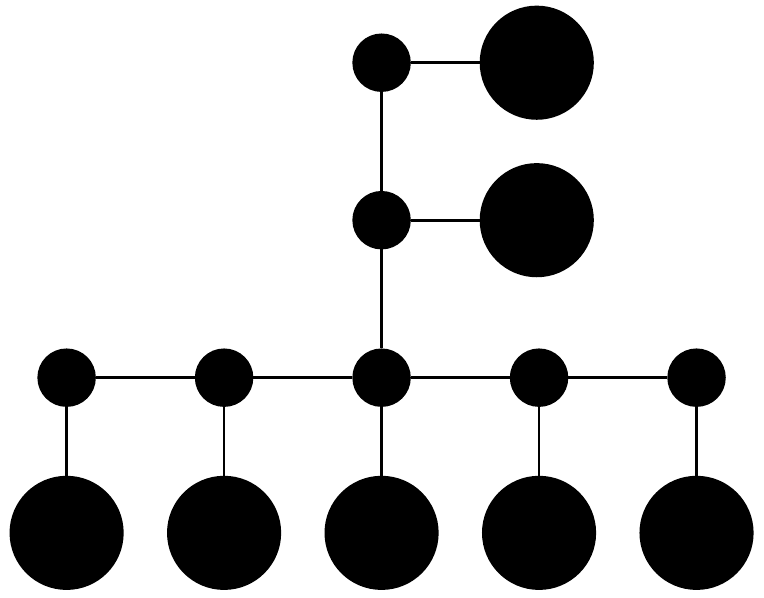}}, \raisebox{-0.5ex}{\includegraphics[scale=0.13]{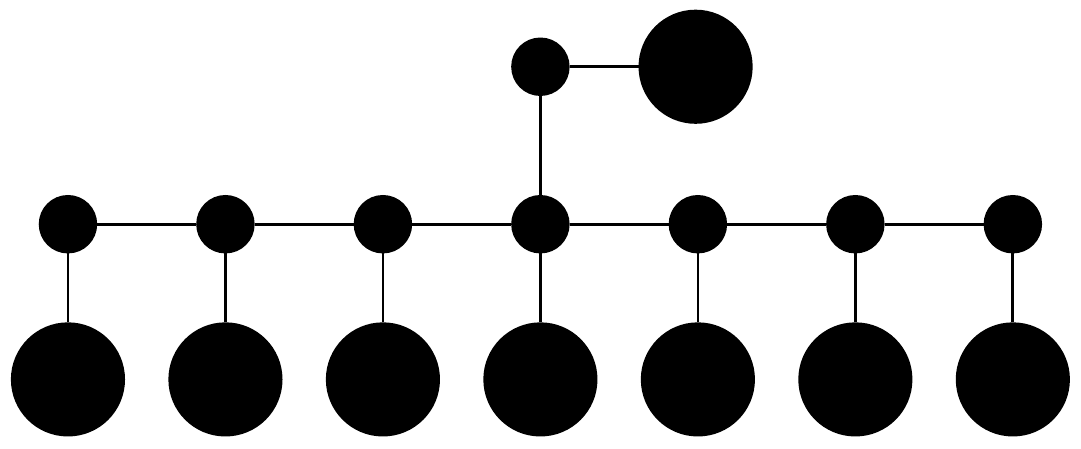}}, \raisebox{-0.5ex}{\includegraphics[scale=0.13]{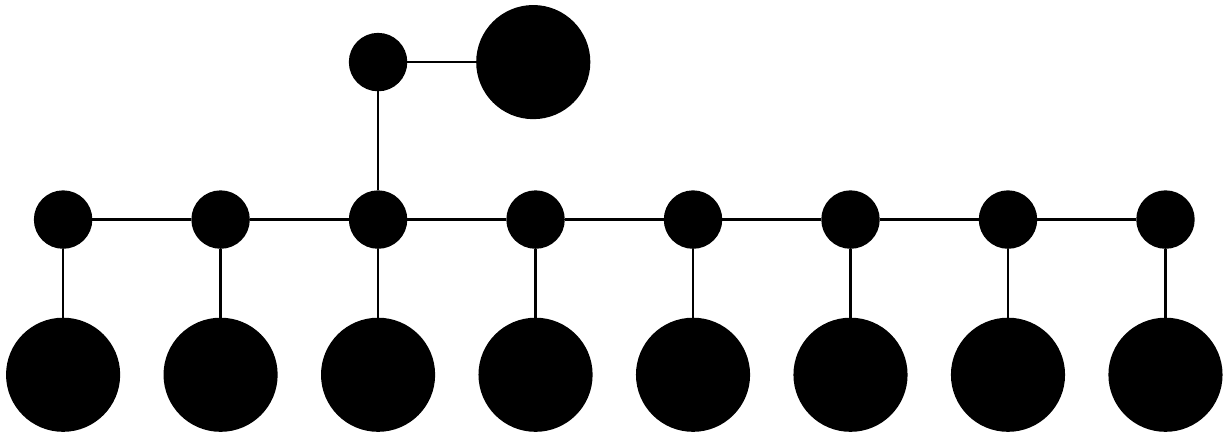}} are also $3$-seedlings. There are other $3$-seedlings, like \raisebox{-0.5ex}{\includegraphics[scale=0.13]{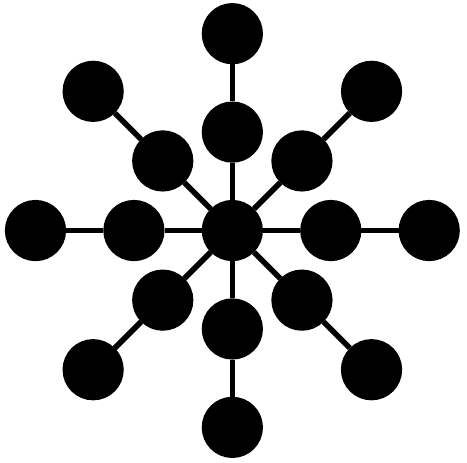}} and \raisebox{-0.5ex}{\includegraphics[scale=0.13]{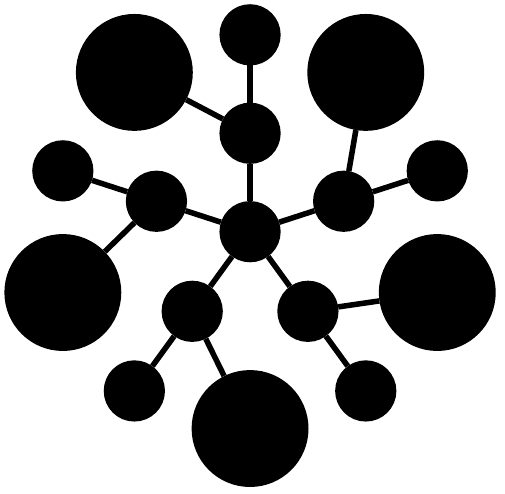}}.

In Theorem \ref{fatseedling} below, we give a classification of the fat $3$-seedlings.
\begin{pro}\label{Esimilar}
For each edge-signed graph $(\widetilde{E}_i,\text{sgn})$ with no two $(-)$-edges incident, there exists a unique fat $3$-seedling having $(\widetilde{E}_i,\text{sgn})$ as its special graph, where $i=6,7,8$.
\end{pro}
\begin{proof}
Let $(\widetilde{E}_i,\text{sgn})$ be an edge-signed graph with no two $(-)$-edges incident. If a fat $3$-seedling having $(\widetilde{E}_i,\text{sgn})$ as its special graph, then it is not integrally representable of norm $3$ by Remark \ref{F-3seedlings} and every slim vertex has exactly one fat neighbor by Theorem \ref{lattice} (ii). Now we define a tree-like Hoffman graph $\mathfrak{t}$, which is obtained by replacing \raisebox{-0.5ex}{\includegraphics[scale=0.8]{seedling3_fat_2}} by \raisebox{-0.5ex}{\includegraphics[scale=0.8]{seedling3_fat_1}} and attaching a fat vertex to each vertex not incident to $(-)$-edge. Clearly, the special graph of $\mathfrak{t}$ is $(\widetilde{E}_i,\text{sgn})$. We claim that $\mathfrak{t}$ is $(-3)$-irreducible, as otherwise $\mathfrak{t}$ is an induced Hoffman subgraph of $\mathfrak{h}$ with $\lambda_{\min}(\mathfrak{h})\geq-3$, which is obtained by attaching fat vertices to $\mathfrak{t}$ and is integrally representable of norm $3$ by Theorem \ref{lattice} (ii). This is not possible as $\mathfrak{t}$ is not integrally representable of norm $3$. We also have $\lambda_{\min}(\mathfrak{t})=-3$, since $Sp(\mathfrak{t})$ is similar to $-I-A(E_i)$, where $A(E_i)$ is the adjacency matrix of $E_i$ by Lemma \ref{-similar} (i). Lemma \ref{3seedling-3} shows that $\mathfrak{t}$ is a $3$-seedling. So we have shown the existence. The uniqueness is obvious and this completes the proof.
\end{proof}

\begin{thm}\label{fatseedling}
Let $\mathfrak{t}$ be a fat $3$-seedling. Then one of the following holds:
 \begin{enumerate}
\item $\mathfrak{t}$ is isomorphic to one of these five Hoffman graphs \raisebox{-0.5ex}{\includegraphics[scale=0.13]{photo1}},
 \raisebox{-0.5ex}{\includegraphics[scale=0.13]{photo2}}, \raisebox{-0.5ex}{\includegraphics[scale=0.13]{photo3}}, \raisebox{-0.5ex}{\includegraphics[scale=0.13]{photo4}}, \raisebox{-0.5ex}{\includegraphics[scale=0.13]{photo5}}.
\item Its special graph $\mathcal{S}(\mathfrak{t})$ is switching equivalent to $\widetilde{E}_6$, $\widetilde{E}_7$ or $\widetilde{E}_8$ and has no two $(-)$-edges incident.
\end{enumerate}
 \end{thm}
\begin{proof}
We may assume that $\mathfrak{t}$ is not integrally representable of norm $3$, as otherwise we are in case (i) by Remark \ref{F-3seedlings}. By Theorem \ref{lattice} (ii), we find that each slim vertex of $\mathfrak{t}$ has exactly one fat neighbor. Let $\mathcal{S}(\mathfrak{t})=(G,\text{sgn})$ be its special graph. Clearly, there is no two $(-)$-edges incident in $\mathcal{S}(\mathfrak{t})$, as otherwise $\mathfrak{t}$ is the Hoffman graph \raisebox{-0.5ex}{\includegraphics[scale=0.13]{photo3}} ($\mathfrak{t}$ is tree-like with $\lambda_{\min}(\mathfrak{t})\geq-3$ and $|N_{\mathfrak{t}}^f(x)|=1$ for all $x\in V_s(\mathfrak{t})$) and this is not possible. Now we focus on $G$. Lemma \ref{connected} shows that $G$ is connected. From Lemma \ref{-similar} (i), we know that there exists a diagonal matrix $D$ such that $DSp(\mathfrak{t})D=-I-A(G)$, where $Sp(\mathfrak{t})$ is the special matrix of $\mathfrak{t}$, $A(G)$ is the adjacent matrix of $G$, and $D_{xx}\in\{1,-1\}$ for all $x\in V_s(\mathfrak{t})$. This implies $\lambda_{\max}(G)=-\lambda_{\min}(\mathfrak{t})-1\leq2$ and $G$ is not integrally representable of norm $2$ (as $\mathfrak{t}$ is not integral representable of norm $3$). Hence $G$ is an induced subgraph of $\widetilde{E}_i$ for some $i\in\{6,7,8\}$, by Smith's Theorem, as mentioned in the introduction. Definition \ref{seedling} (ii) and Proposition \ref{Esimilar} show that $G$ is exactly $\widetilde{E}_i$ for some $i\in\{6,7,8\}$. This completes the proof.
\end{proof}

\begin{re}\label{number}
There are $7$ fat $3$-seedlings with its special graph switching equivalent to $\widetilde{E}_6$; $18$ fat $3$-seedlings with its special graph switching equivalent to $\widetilde{E}_7$; $50$ fat $3$-seedlings with its special graph switching equivalent to $\widetilde{E}_8$.
\end{re}

We would like to conclude this paper with the following problems:

\begin{problem}
Classify the $3$-seedlings.
\end{problem}

A subproblem of {\bf Problem 1} is the following problem.
\begin{problem}
Classify the $3$-seedlings $\mathfrak{t}$ such that there exists a reduced representation $\psi: V_s(\mathfrak{t})\rightarrow \mathbb{R}^n$ of norm $3$ for some $n$, satisfying $2\psi(x)\in \mathbb{Z}^n$ for all $x\in V_s(\mathfrak{t})$.
\end{problem}

\section*{Address:}
\textbf{a)}   Wen-Tsun Wu Key Laboratory of CAS, School of Mathematical Sciences, University
of Science and Technology of China, Hefei, Anhui, 230026, P.R. China \\

\noindent \textbf{b)} School of Mathematical Sciences, University of Science and Technology of China,
Hefei, Anhui, 230026, P.R. China\\

\noindent \textbf{b)} School of Mathematical Sciences, University of Science and Technology of China,
Hefei, Anhui, 230026, P.R. China\\
\section*{Email Address:}
%\\Email:
koolen@ustc.edu.cn (J. H. Koolen)
\newline
masoodqau27@gmail.com; masood@mail.ustc.edu.cn (M. U. Rehman)
\newline
xuanxue@mail.ustc.edu.cn (Q. Yang)

\begin{thebibliography}{99}
\bibitem{drg} A.E. Brouwer, A.M. Cohen, A. Neumaier, \emph{Distance-Regular Graphs}, Springer-Verlag, Berlin, 1989.
\bibitem{PJC} P.J. Cameron, J.M. Goethals, J.J. Seidel and E.E. Shult, Line graphs, root systems and elliptic geometry, \emph{J. Algebra} \textbf{43} (1976), 305--327.
\bibitem{CS} J. H. Conway, N. J. A. Sloane, Complex and integral laminated lattices, \emph{Trans. Amer. Math. Soc.} {\bf 280} (1983), 463--490.
\bibitem{DC} D. Cvetkovi\'{c}, P. Rowlinson and S. K. Simi\'{c}, \emph{Spectral Generalizations of Line Graphs--On Graphs with Least Eigenvalue $-2$}, Cambridge Univ. Press, Cambridge, 2004.
\bibitem{H1995} W.H. Haemers, Interlacing eigenvalues and graphs, \emph{Linear Algebra Appl.} {\bf 226--228} (1995), 593--616.
\bibitem{AJH} A.J. Hoffman, On graphs whose least eigenvalue exceeds $-1-\sqrt{2}$, \emph{Linear Algebra Appl.} \textbf{16} (1977), 153--165.
\bibitem{Koolen} H.J. Jang, J. Koolen, A. Munemasa and T. Taniguchi, On fat Hoffman graphs with smallest eigenvalue at least $-3$, \emph{Ars Math. Contemp.} \textbf{7} (2014), 105--121.
\bibitem{KLY} J.H. Koolen, Y.-R. Li and Q. Yang, On fat Hoffman graphs with smallest eigenvalue at least $-3$-\uppercase\expandafter{\romannumeral2}, in preparation.
\bibitem{Neumaier} A. Neumaier, On Norm Three Vectors in Integral Euclidean Lattices. I, \emph{Math. Z.} {\bf 183} (1983), 565--574.
\bibitem{PP} W. Plesken, M. Pohst, Constructing integral lattices with prescribed minimum. I. \emph{Math. Comp.} {\bf 45} (1985), 209--221.
\bibitem{PP2} W. Plesken, M. Pohst, Constructing integral lattices with prescribed minimum. II, \emph{Math. Comp.} {\bf 60} (1993), 817--825.
\bibitem{SH} R. Scharlau, B. Hemkemeier, Classification of integral lattices with large class number, \emph{Math. Comp.} {\bf 67} (1998), 737--749.
\bibitem{smith} J. H. Smith, Some properties of the spectrum of a graph, pp. 403--406 in: \emph{Combinatorial Structures and their Applications} (Proc. Conf. Calgary 1969) (R. Guy et al., eds.), Gordon and Breach, New York, 1970.
\bibitem{woo} R. Woo and A. Neumaier, On graphs whose smallest eigenvalue is at least $-1-\sqrt{2}$, \emph{Linear Algebra Appl.} \textbf{226--228} (1995), 577--591.
\end{thebibliography}
\end{document}